\documentclass[12pt,reqno]{amsart}
\usepackage{amsmath,amsthm,amsfonts,amssymb,times}
\usepackage{verbatim}
\usepackage{url}
\setbox0=\hbox{$+$}
\newdimen\plusheight
\plusheight=\ht0
\def\+{\;\lower\plusheight\hbox{$+$}\;}

\setbox0=\hbox{$-$}
\newdimen\minusheight
\minusheight=\ht0
\def\-{\;\lower\minusheight\hbox{$-$}\;}

\setbox0=\hbox{$\cdots$}
\newdimen\cdotsheight
\cdotsheight=\plusheight%\ht0
\def\cds{\lower\cdotsheight\hbox{$\cdots$}}

\makeatletter
\def\leqalignno#1{\displ@y \tabskip\z@ plus\@ne fil
  \halign to\displaywidth{\hfil$\@lign\displaystyle{##}$\tabskip\z@skip
    &$\@lign\displaystyle{{}##}$\hfil\tabskip\z@ plus\@ne fil
    &\kern-\displaywidth\rlap{$\@lign\hbox{\rm##}$}\tabskip\displaywidth\crcr
    #1\crcr}}
\makeatother
\let\dotlessi=\i

\newcommand{\eb}{\begin{equation}}
\newcommand{\ee}{\end{equation}}

\newcommand{\df}{\dfrac}
\newcommand{\tf}{\tfrac}

\renewcommand{\d}{{\delta}}

\renewcommand{\(}{\left\(}
\renewcommand{\)}{\right\)}
\renewcommand{\[}{\left\[}
\renewcommand{\]}{\right\]}
\renewcommand{\i}{\infty}
\renewcommand{\pmod}[1]{\,(\textup{mod}\,#1)}
\numberwithin{equation}{section}
 \theoremstyle{plain}
\newtheorem{theorem}{Theorem}[section]
\newtheorem{lemma}[theorem]{Lemma}

\newtheorem{entry}[theorem]{Entry}
\newtheorem{conjecture}[theorem]{Conjecture}
\newtheorem{remark}[theorem]{Remark}
\newtheorem{definition}[theorem]{Definition}

\DeclareMathOperator{\Realp}{Re}
\DeclareMathOperator{\Imp}{Im}

\begin{document}
\title[Trigonometric and Bessel Series]
{Balanced Derivatives, Identities, and Bounds for Trigonometric and Bessel Series}
%first attempts at a title and abstract -- probably too long
\author{Bruce C.~Berndt, Martino Fassina, Sun Kim, and Alexandru Zaharescu}
\address{Department of Mathematics, University of Illinois, 1409 West Green
Street, Urbana, IL 61801, USA} \email{berndt@illinois.edu}
\address{Fakult\"at f\"ur Mathematik, Universit\"at Wien, Oskar-Morgenstern-Platz 1, 1090 Wien, Austria}
\email{martino.fassina@univie.ac.at}
\address{Department of Mathematics, and Institute of Pure and Applied Mathematics, Jeonbuk National University, 567 Baekje-daero, Jeonju-si, Jeollabuk-do 54896, Republic of Korea}
\email{sunkim@jbnu.ac.kr}
\address{Department of Mathematics, University of Illinois, 1409 West Green
Street, Urbana, IL 61801, USA; Institute of Mathematics of the Romanian
Academy, P.O.~Box 1-764, Bucharest RO-70700, Romania}
\email{zaharesc@illinois.edu}

\begin{abstract}
Motivated by two identities published with Ramanujan's lost notebook and connected, respectively, with the Gauss circle problem and the Dirichlet divisor problem, in an earlier paper, three of the present authors derived representations for certain sums of products of trigonometric functions as double series of Bessel functions. These series are generalized in the present paper by introducing the novel notion of balanced derivatives, leading to further theorems.   As we will see below, the regions of convergence in the unbalanced case are entirely different than those in the balanced case. From this viewpoint, it is remarkable that Ramanujan had the intuition to formulate entries that are, in our new terminology, ``balanced''.  If $x$ denotes the number of products of the trigonometric functions appearing in our sums, in addition to proving the identities mentioned above, theorems and conjectures for upper and lower bounds for the sums as $x\to\infty$ are established.
\end{abstract}
\maketitle

\section{Introduction and main results}
In a series of papers \cite{besselII}, \cite{bessie4}, \cite{bessie3},
\cite{bessel2bbskaz}, \cite{bessel1} written by three of the present authors and J.~Li, they examined two formulas of Ramanujan in an unpublished fragment found in \cite[p.~335]{lnb}.  The two formulas are connected with the famous \emph{Gauss circle problem} and the equally famous \emph{Dirichlet divisor problem}.  Each of the two formulas has three distinct interpretations.  Ramanujan's formulas and the methods developed to prove them have generated further research, in particular, in \cite{bessie4} and \cite{bessel2bbskaz}. In this paper, we continue our study by examining ``balanced'' derivatives of the series representations and making applications to the trigonometric sums studied in \cite{bessel2bbskaz}.

In order to state Ramanujan's formulas, the \emph{Gauss circle problem}, and the \emph{Dirichlet divisor problem}, it is necessary to first define the relevant Bessel functions appearing in Ramanujan's identities.  Let $J_{\nu}(z)$ denote the ordinary Bessel function of order $\nu$.  Define
\begin{equation}\label{defI}
I_{\nu}(z) := -Y_{\nu}(z)-\df{2}{\pi}K_{\nu}(z),
\end{equation}
  where $Y_{\nu}(z)$ denotes the Bessel function of imaginary argument of order $\nu$ given by
\begin{equation}\label{b.Y}
Y_{\nu}(z):=\df{J_{\nu}(z)\cos(\nu\pi)-J_{-\nu}(z)}{\sin(\nu\pi)}, \quad |z|<\infty,
\end{equation}
and $K_{\nu}(z)$ denotes the modified Bessel function of order $\nu$ defined by
\begin{equation}\label{b.K}
K_{\nu}(z):=\df{\pi}{2}\,\df{e^{\pi{i}\nu/2}J_{-\nu}(iz)-e^{-\pi{i}\nu/2}J_{\nu}(iz)}{\sin(\nu\pi)},\quad -\pi<\arg z<\tf12\pi.
\end{equation}
If $\nu$ is an integer $n$, it is understood that we define
the functions by taking the limits as $\nu\to{n}$ in \eqref{b.Y}
and \eqref{b.K}.

We now describe the \emph{Gauss circle problem} and the \emph{Dirichlet divisor problem}.  Detailed discussions and references for these two famous problems can be found in \cite{monthly}.

  Let $r_2(n)$ denote the number of ways in which the positive integer $n$ can be expressed as a sum of two squares, where different orders and different signs of the summands are regarded as distinct representations of $n$ as a sum of two squares. For example, $5=(\pm2)^2+(\pm1)^2=(\pm1)^2+(\pm2)^2$, and so $r_2(5)=8$. Let
  \begin{equation*}\label{r2sum}
  R(x):={\sum_{0\leq n\leq x}}^{\prime}r_2(n),
  \end{equation*}
  where $r_2(0)=1$ and the prime $'$ indicates that, if $n=x$, only $\tf12r_2(x)$ is counted. Then
  \begin{equation}\label{b.circle}
R(x)=: \pi x +P(x) = \pi x +
\sum_{n=1}^{\i}r_2(n)\left(\df{x}{n}\right)^{1/2}J_1(2\pi\sqrt{nx}),
\end{equation}
where the representation for $P(x)$ on the far right side is due to Hardy and Ramanujan \cite[p.~265]{hardy1916}.
Finding the precise order of magnitude of $P(x)$, as $x\to\infty$, is known as the \emph{Gauss circle problem}.

 Throughout this paper, for arithmetic functions $f$ and $g$, let
 \begin{equation*}
 {\sum_{nm\leq x}}^{\prime} f(n)g(m):=\begin{cases}
\sum_{nm\leq x} f(n)g(m),\quad\quad \quad\quad \quad\quad \quad\quad \quad\,\text{ if $x$ is not an integer,}\\
\sum_{nm\leq x} f(n)g(m)-\frac{1}{2}\sum_{nm= x} f(n)g(m), \text{ if $x$ is an integer.}\\
\end{cases}
 \end{equation*}

   In \cite{bessie3}, three of the authors proved the following enigmatic identity of Ramanujan from his lost notebook \cite{lnb}.
 % To state Ramanujan's identity, we need to first define
 % \begin{equation*}
 % F(x):=\begin{cases} [x],\quad\quad \,\,\,\text{if $x$ is not an integer,} %\\
%  x-\frac{1}{2}, \quad \text{ if $x$ is an integer.}
 % \end{cases}
 % \end{equation*}

\begin{entry}\cite[p.~335]{lnb}\label{besselseries} %Let $J_{\nu}(z)$ denote the ordinary Bessel function of order $\nu$.
If $0 < \theta < 1$ and $x>0$, then
\begin{multline}\label{b1.1}
{\sum_{n\leq x}}^{\prime}\left[\frac{x}{n}\right]\sin(2\pi n\theta)= \pi
x\left(\dfrac{1}{2}-\theta\right) -\df{1}{4}\cot(\pi\theta)\\
+\frac{1}{2}\sqrt{x}\sum_{m=1}^{\i}\sum_{n=0}^{\infty}
\left\{\df{J_1\left(4\pi\sqrt{m(n+\theta)x}\right)}{\sqrt{m(n+\theta)}}
-\df{J_1\left(4\pi\sqrt{m(n+1-\theta)x}\right)}{\sqrt{m(n+1-\theta)}}\right\},
\end{multline}
where,  as customary, $[x]$ denotes the greatest integer less than or equal to $x$.
\end{entry}

The identity \eqref{b1.1} can be seen as a 2-variable analogue of \eqref{b.circle}. If we set $\theta=\tf14$, by an elementary theorem on $r_2(n)$ \cite[p.~313]{hardywright}, the left sides of \eqref{b.circle} and \eqref{b1.1} are identical.

Let $d(n)$ denote the number of positive divisors of the positive integer $n$, and let
\begin{equation*}\label{Dx}
   D(x):={\sum_{n\leq x}}^{\prime}d(n).
   \end{equation*}
 Then
  \begin{align}
D(x) =&x\left(\log x
+(2\gamma-1)\right)+\Delta(x)\label{elemdirichlet}\\=&x\left(\log x
+(2\gamma-1)\right)
+\frac{1}{4}+\sum_{n=1}^{\i}d(n)
\left(\df{x}{n}\right)^{1/2}I_1(4\pi\sqrt{nx}),\label{vor}
\end{align}
where $x>0$, $\gamma$ denotes Euler's constant, the identity \eqref{elemdirichlet} is due to Dirichlet \cite{dirichlet} and defines the ``error term'' $\Delta(x)$, $I_1(x)$ is defined in \eqref{defI},  and the series representation for $\Delta(x)$ on the  right side of \eqref{vor} is due to Vorono\"{\dotlessi} \cite{voronoi}. Finding the optimal bound for $\Delta(x)$ as $x\to\infty$ is the \emph{Dirichlet divisor problem}.
Vorono\"{\dotlessi} \cite{voronoi} proved that
\begin{equation}\label{d(n)bigO}
\Delta(x)=O(x^{1/3}\log x).
\end{equation}
The upper bound for $\Delta(x)$ given in \eqref{d(n)bigO} is not the best currently known.  See \cite{monthly} for a list of upper bounds that have been obtained for $\Delta(x)$.  Furthermore \cite{hardy1916}, \cite{ingham}, \cite[p.~130]{annals2},
\begin{equation}\label{d(n)omega}
\Delta(x)=\Omega_{\pm}(x^{1/4}),
\end{equation}
as $x\to\infty$. We say that $f(x)=\Omega_+(x^{\theta})$ if there exists a sequence $\{x_n\}\to\infty$ such that
$$f(x_n)\leq C_1 (x_n)^{\theta}$$
fails to hold for every positive constant $C_1$. Similarly,
$
f(x)=\Omega_-(x^{\theta})$
if there exists a sequence $\{x_n^{'}\}\to\infty$ such that
$$f({x_n^{'}})\geq -C_2(x_n^{'})^{\theta}$$
fails to hold for every positive constant $C_2$.

A proof of the following second enigmatic identity of Ramanujan from \cite[p.~335]{lnb} has been given by J.~Li and two of the present authors \cite{Paper2}. When $\theta=0$, the left-hand side of \eqref{b.c.11} below reduces to the left-hand side of \eqref{vor}.

 \begin{entry}\cite[p.~335]{lnb}\label{b.besselseries2}  For  $x>0$ and $0<\theta<1$,
\begin{align}\label{b.c.11}
&{\sum_{n\leq x}}^{\prime}\left[\frac{x}{n}\right]\cos(2\pi n\theta)=
\df{1}{4} -x\log(2\sin(\pi\theta)) \\
&+\df{1}{2}\sqrt{x}\sum_{m=1}^{\i}\sum_{n=0}^{\i}
\left\{\df{I_1\left(4\pi\sqrt{m(n+\theta)x}\right)}{\sqrt{m(n+\theta)}}
+\df{I_1\left(4\pi\sqrt{m(n+1-\theta)x}\right)}{\sqrt{m(n+1-\theta)}}\right\}.\notag
\end{align}
\end{entry}

\begin{remark}\label{keyremark} Examining \eqref{vor} and \eqref{b.c.11},  we observe  that ``big O" conjectures and theorems about the error term $\Delta(x)$, which traditionally and frequently involve the series of Bessel functions on the right-hand side of \eqref{vor}, pertain to the double series of Bessel functions on the right-hand side of \eqref{b.c.11}, although, as of the present, \eqref{b.c.11} has not been employed in deriving ``big O'' theorems.
\end{remark}

 In this paper we prove new identities in the spirit of Ramanujan, where on the left sides are sums of products of two trigonometric functions, while on the right sides are double series of Bessel functions. Our formulas involve two parameters $\sigma,\theta$ in the interval $(0,1)$ and stem from identities that three of the present authors proved in \cite{bessel2bbskaz}. The novelty in the current paper consists in the possibility of taking termwise derivatives with respect to $\sigma$ and $\theta$. Since we are only allowed (by reasons of convergence) to take the same number of derivatives in $\sigma$ and in $\theta$, we say that the identities thus obtained are ``balanced". As an interesting application, we consider two identities that were independently proved in \cite{bessel2bbskaz}, and we show that one of them can be obtained as the first balanced derivative of the other (see Section \ref{last}).

 In the second portion of the paper,  our goal is to derive ``big O'' and $\Omega$ theorems for sums of two trigonometric functions.  Unfortunately, in many cases, for lower bounds we are only able to make conjectures. We also extend our study of  sine sums
 $${\sum_{mn\leq x}}^{\prime}mn\sin(2\pi na/p)\sin(2\pi mb/q)$$
 to sums of $k$ sines, $k\geq 2$.
 %are we able to obtain a non-trivial ``big O'' theorem, but we can also establish a ``big O'' theorem for
 %$$\sideset{}{'}\sum_{1\le n_1n_2\cdots n_k\le x} n_1n_2\cdots n_k\,\sin(2\pi n_1a_1/p_1)\sin(2\pi n_2a_2/p_2)\cdots \sin(2\pi n_ka_k/p_k). $$

\section{Identities for Trigonometric Sums in Terms of Bessel Functions}
 Here is our first main result.
 \begin{theorem}\label{BI}
Let $\sigma,\theta$ be in the interval $(0,1)$, and let $x>0$. Then for every non-negative integer $k$,
\begin{equation}\label{ii}
\begin{split}
&\frac{\partial^{2k}}{\partial\sigma^{k}\partial\theta^{k}}
\Bigg{\{}{\sum_{mn\leq x}}^{\prime}\cos(2\pi m\sigma)\sin(2\pi n\theta)+\frac{\cot(\pi\theta)}{4}\Bigg{\}}\\=&\frac{\sqrt{x}}{4}\sum_{m=0}^{\infty}
\sum_{n=0}^{\infty}\frac{\partial^{2k}}{\partial\sigma^{k}\partial\theta^{k}}
\Bigg{\{} \frac{J_1(4\pi\sqrt{(m+\sigma)(n+\theta)x})}{\sqrt{(m+\sigma)(n+\theta)}}
+\frac{J_1(4\pi\sqrt{(m+1-\sigma)(n+\theta)x})}{\sqrt{(m+1-\sigma)(n+\theta)}}\\
&-\frac{J_1(4\pi\sqrt{(m+\sigma)(n+1-\theta)x})}{\sqrt{(m+\sigma)(n+1-\theta)}}
-\frac{J_1(4\pi\sqrt{(m+1-\sigma)(n+1-\theta)x})}{\sqrt{(m+1-\sigma)(n+1-\theta)}}
 \Bigg{\}}.
 \end{split}
\end{equation}
 \end{theorem}

The proof of Theorem \ref{BI} relies on a detailed study of the convergence of a double series more general than the one appearing on the right-hand side of \eqref{ii}.

Let $\sigma,\theta, x$ be as in Theorem \ref{BI}. Let $\alpha,\beta$ be non-negative integers and $s,w$ complex numbers. Consider the double series
 \begin{equation}\label{fun}
 \begin{split}
&G^{\alpha,\beta}(x,\sigma,\theta,s,w):=\sum_{m=0}^{\infty}\sum_{n=0}^{\infty}
\frac{\partial^{\alpha+\beta}}{\partial\sigma^{\alpha}\partial\theta^{\beta}}\Bigg{\{} \frac{J_1(4\pi\sqrt{(m+\sigma)(n+\theta)x})}{(m+\sigma)^s(n+\theta)^w}\\
&+\frac{J_1(4\pi\sqrt{(m+1-\sigma)(n+\theta)x})}{(m+1-\sigma)^s(n+\theta)^w}
-\frac{J_1(4\pi\sqrt{(m+\sigma)(n+1-\theta)x})}{(m+\sigma)^s(n+1-\theta)^w}\\
&-\frac{J_1(4\pi\sqrt{(m+1-\sigma)(n+1-\theta)x})}{(m+1-\sigma)^s(n+1-\theta)^w}
 \Bigg{\}}.
\end{split}
 \end{equation}
We determine values of $\alpha,\beta, s, w$ for which the double series $G^{\alpha,\beta}(x,\sigma,\theta,s,w)$ converges.

\begin{theorem}\label{M}
Let $G^{\alpha,\beta}(x,\sigma,\theta,s,w)$ be defined as above. Assume that
\begin{equation*}
\begin{cases}
4\Realp(s)+2\alpha-2\beta>1,\\
4\Realp(w)+2\beta-2\alpha>1.
\end{cases}
\end{equation*}
Moreover, if $x$ is an integer, assume that $\Realp(s)+\Realp(w)>\frac{25}{26}$, while if $x$ is not an integer, assume that $\Realp(s)+\Realp(w)>\frac{5}{6}$. Then the double series $G^{\alpha,\beta}(x,\sigma,\theta,s,w)$ converges uniformly with respect to $\sigma$ and $\theta$ in any compact subset of $(0,1)^2$.
\end{theorem}

Consider the interesting case $s=\frac{1}{2}=w$, which corresponds to the setting of Theorem \ref{BI}. To meet the conditions of Theorem \ref{M} ensuring the convergence of \eqref{fun}, the only possibility is to take $\alpha=\beta=k$ (``balanced" situation).

\begin{remark}
 Theorem \ref{M} also shows, for every choice of non-negative integers $\alpha,\beta$, that there exists an unbounded region $D_{\alpha,\beta}$ of $\mathbb{C}^2$ such that for every $(s,w)\in D_{\alpha,\beta}$, the corresponding series \eqref{fun} converges.
\end{remark}

  %(In \cite[Page 71, Theorem 2.2]{bessel2bbskaz}, the sum on $m$ should begin with $m=1$ and not $m=0$.)
With similar methods as the those used to prove Theorem \ref{BI}, we establish other ``balanced" identities similar to \eqref{ii}. In these new identities the left-hand side contains only cosines or only sines, respectively.

	First recall, for each integer $\nu$, that $I_{\nu}$ was defined in \eqref{defI}.
 %\begin{equation}\label{defI}
 %I_{\nu}(z):=-Y_{\nu}(z)-\frac{2}{\pi}K_{\nu}(z).
 %\end{equation}
% Here $Y_{\nu}(z)$ is the Bessel function of the second kind of order $\nu$ \cite[page 64]{watson}, and $K_{\nu}$ is the modified Bessel function of order $\nu$ \cite[page 78]{watson}.
Let then
 \begin{equation}\label{t32}
 T_{\frac{3}{2}}(x):=\int_{0}^{\infty} J_{\frac{1}{2}}(u) J_{\frac{3}{2}}\Big(\frac{x}{u}\Big)du.
 \end{equation}
  (In \cite[p. 71]{bessel2bbskaz}, the definition \eqref{t32} is misprinted; replace $J_{\frac32}(x)$ by $J_{\frac{3}{2}}(\tfrac{x}{u})$ there.)

 \begin{theorem}\label{TI}
Let $\sigma,\theta$ be in the interval $(0,1)$, and let $x>0$. Then for every non-negative integer $k$,
\begin{equation}\label{funI}
\begin{split}
&\frac{\partial^{2k}}{\partial\sigma^{k}\partial\theta^{k}}\Bigg{\{}{\sum_{mn\leq x}}^{\prime}\cos(2\pi m\sigma)\cos(2\pi n\theta)-\frac{1}{4}\Bigg{\}}\\=&\frac{\sqrt{x}}{4}\sum_{m=0}^{\infty}\sum_{n=0}^{\infty}\frac{\partial^{2k}}{\partial\sigma^{k}\partial\theta^{k}}
\Bigg{\{} \frac{I_1(4\pi\sqrt{(m+\sigma)(n+\theta)x})}{\sqrt{(m+\sigma)(n+\theta)}}
+\frac{I_1(4\pi\sqrt{(m+1-\sigma)(n+\theta)x})}{\sqrt{(m+1-\sigma)(n+\theta)}}\\
&+\frac{I_1(4\pi\sqrt{(m+\sigma)(n+1-\theta)x})}{\sqrt{(m+\sigma)(n+1-\theta)}}
+\frac{I_1(4\pi\sqrt{(m+1-\sigma)(n+1-\theta)x})}{\sqrt{(m+1-\sigma)(n+1-\theta)}}
 \Bigg{\}}.
 \end{split}
\end{equation}
 \end{theorem}
 \begin{theorem}\label{TT}
Let $\sigma,\theta$ be in the interval $(0,1)$, and let $x>0$. Then for every non-negative integer $k$,
\begin{equation}\label{funT}
\begin{split}
&\frac{\partial^{2k}}{\partial\sigma^{k}\partial\theta^{k}}\Bigg{\{}{\sum_{mn\leq x}}^{\prime}mn\sin(2\pi m\sigma)\sin(2\pi n\theta)\Bigg{\}}\\=&\frac{x\sqrt{x}}{4}\sum_{m=0}^{\infty}\sum_{n=0}^{\infty}\frac{\partial^{2k}}{\partial\sigma^{k}\partial\theta^{k}}
\Bigg{\{} \frac{T_{\frac{3}{2}}(4\pi^2(m+\sigma)(n+\theta)x)}{\sqrt{(m+\sigma)(n+\theta)}}
-\frac{T_{\frac{3}{2}}(4\pi^2(m+1-\sigma)(n+\theta)x)}{\sqrt{(m+1-\sigma)(n+\theta)}}\\
&-\frac{T_{\frac{3}{2}}(4\pi^2(m+\sigma)(n+1-\theta)x)}{\sqrt{(m+\sigma)(n+1-\theta)}}
+\frac{T_{\frac{3}{2}}(4\pi^2(m+1-\sigma)(n+1-\theta)x)}{\sqrt{(m+1-\sigma)(n+1-\theta)}}
 \Bigg{\}}.
 \end{split}
\end{equation}
 \end{theorem}

\section{Convergence of ``almost balanced" double series}
 This section is devoted to the proof of Theorem \ref{M}. We start with a simple calculation.
\begin{lemma}\label{L1}
The following identity holds:
\begin{equation}\label{pr}
\frac{\partial^{\alpha+\beta}}{\partial\sigma^{\alpha}\partial\theta^{\beta}} \frac{J_1(4\pi\sqrt{(m+\sigma)(n+\theta)x})}{(m+\sigma)^s(n+\theta)^w}=\frac{\sum_{\nu\in A}c_{\nu} J_{\nu}(4\pi\sqrt{(m+\sigma)(n+\theta)x})}{(m+\sigma)^{s+\frac{\alpha}{2}
-\frac{\beta}{2}}(n+\theta)^{w+\frac{\beta}{2}-\frac{\alpha}{2}}}+\cdots.
\end{equation}
Here $A$ is a finite subset of $\mathbb{Z}$, the constants $c_{\nu}$ are non-negative, and the dots $\cdots$ denote a sum of terms that have the form
\begin{equation}\label{p}  r_{\nu}\,\frac{J_{\nu}(4\pi\sqrt{(m+\sigma)(n+\theta)x})}{(m+\sigma)^{\gamma}(n+\theta)^{\delta}},\quad \nu\in\mathbb{Z},\, r_{\nu}\in\mathbb{R},
\end{equation}
where
\begin{equation}\label{o}
\begin{cases}
\Realp(\gamma)\geq \Realp(s)+\frac{\alpha}{2}-\frac{\beta}{2},\\
\Realp(\delta)\geq \Realp(w)+\frac{\beta}{2}-\frac{\alpha}{2},\\
\end{cases}
\end{equation}
with at least one of the two inequalities in \eqref{o} being strict.
\end{lemma}
\begin{proof}
Recall from \cite[p. 17]{watson} the following identity, which holds for every integer $\nu$: \begin{equation}\label{mara}
J_{\nu-1}(z)-J_{\nu+1}(z)=2J'_{\nu}(z).
\end{equation}
We argue by induction on the total number of derivatives $k=\alpha+\beta$. For $k=0$ there is nothing to prove. Assume now that the statement holds for some $k\geq 0$. We will prove that it holds for $k+1$. Write $k+1=\alpha+\beta$, and assume without loss of generality that $\alpha\geq 1$. By the inductive hypothesis,
\begin{equation}\label{ind}
\frac{\partial^{k+1}}{\partial\sigma^{\alpha}\partial\theta^{\beta}} \frac{J_1(4\pi\sqrt{(m+\sigma)(n+\theta)x})}{(m+\sigma)^s(n+\theta)^w}=\frac{\partial}{\partial \sigma}\Bigg{[}\frac{\sum_{\nu\in A}c_{\nu} J_{\nu}(4\pi\sqrt{(m+\sigma)(n+\theta)x})}{(m+\sigma)^{s+\frac{\alpha-1}{2}
-\frac{\beta}{2}}(n+\theta)^{w+\frac{\beta}{2}-\frac{\alpha-1}{2}}}+\cdots\Bigg{]}.
\end{equation}
The dots $\cdots$ in \eqref{ind} stand for terms of the form \eqref{p}, with the exponents $\gamma$ and $\delta$ satisfying
\begin{equation}\label{coo}
\begin{cases}
\Realp(\gamma)\geq \Realp(s)+\frac{\alpha-1}{2}-\frac{\beta}{2},\\
\Realp(\delta)\geq \Realp(w)+\frac{\beta}{2}-\frac{\alpha-1}{2},\\
\end{cases}
\end{equation}
where at least one of the two inequalities is strict. By the chain rule and \eqref{mara},
\begin{gather}%\label{eins}
%\begin{split}
\frac{\partial}{\partial\sigma} \Bigg{[}\frac{J_{\nu}(4\pi\sqrt{(m+\sigma)(n+\theta)x})}{(m+\sigma)^{s+\frac{\alpha-1}{2}
-\frac{\beta}{2}}(n+\theta)^{w+\frac{\beta}{2}-\frac{\alpha-1}{2}}}\Bigg{]}
=\pi\sqrt{x}\,\frac{(J_{\nu-1}-J_{\nu+1})(4\pi\sqrt{(m+\sigma)(n+\theta)x})}{(m+\sigma)^{s+\frac{\alpha}{2}
-\frac{\beta}{2}}(n+\theta)^{w+\frac{\beta}{2}-\frac{\alpha}{2}}}\notag\\+\Big(\frac{\beta}{2}
-\frac{\alpha-1}{2}-s\Big)\frac{J_{\nu}(4\pi\sqrt{(m+\sigma)(n+\theta)x})}{(m+\sigma)^{s+\frac{\alpha+1}{2}
-\frac{\beta}{2}}(n+\theta)^{w+\frac{\beta}{2}-\frac{\alpha-1}{2}}}.\label{eins}
%\end{split}
\end{gather}
 Similarly, for the ``error terms" in \eqref{ind},
\begin{equation}\label{zwei}
\begin{split}
\frac{\partial}{\partial\sigma} \Bigg{[}\frac{J_{\nu}(4\pi\sqrt{(m+\sigma)(n+\theta)x})}{(m+\sigma)^{\gamma}(n+\theta)^{\delta}}\Bigg{]}
=&\pi\sqrt{x}\,\frac{(J_{\nu-1}-J_{\nu+1})(4\pi\sqrt{(m+\sigma)(n+\theta)x})}{(m+\sigma)^{\gamma+\frac{1}{2}}
(n+\theta)^{\delta-\frac{1}{2}}}\\&-\gamma\,
\frac{J_{\nu}(4\pi\sqrt{(m+\sigma)(n+\theta)x})}{(m+\sigma)^{\gamma+1}(n+\theta)^{\delta}}.
\end{split}
\end{equation}
The second term on the right side of \eqref{eins} and both terms on the right side of \eqref{zwei} are of the form
\begin{equation*}  r_{\nu}\,\frac{J_{\nu}(4\pi\sqrt{(m+\sigma)(n+\theta)x})}{(m+\sigma)^{\gamma_1}(n+\theta)^{\delta_1}},\quad \nu\in\mathbb{Z},\, r_{\nu}\in\mathbb{R},
\end{equation*}
where by \eqref{coo} the complex numbers $\gamma_1$ and $\delta_1$ satisfy
\begin{equation*}
\begin{cases}
\Realp(\gamma_1)\geq \Realp(s)+\frac{\alpha}{2}-\frac{\beta}{2},\\
\Realp(\delta_1)\geq \Realp(w)+\frac{\beta}{2}-\frac{\alpha}{2},\\
\end{cases}
\end{equation*}
with at least one of the two inequalities being strict. Substituting the identities \eqref{eins} and \eqref{zwei} into \eqref{ind}, we obtain \eqref{pr}, thus completing the proof.
\end{proof}

The same analysis can be repeated for the other summands appearing in \eqref{fun}, separating in each case the ``main term" from the ``error terms." We then obtain identities analogous to \eqref{pr}. For instance,
\begin{gather}%\label{pro}
%\begin{split}
\frac{\partial^{\alpha+\beta}}{\partial\sigma^{\alpha}\partial\theta^{\beta}} \frac{J_1(4\pi\sqrt{(m+1-\sigma)(n+\theta)x})}{(m+1-\sigma)^s(n+\theta)^w}\notag\\=(-1)^{\alpha}\frac{\sum_{\nu\in A}c_{\nu} J_{\nu}(4\pi\sqrt{(m+1-\sigma)(n+\theta)x})}{(m+1-\sigma)^{s+\frac{\alpha}{2}
-\frac{\beta}{2}}(n+\theta)^{s+\frac{\beta}{2}-\frac{\alpha}{2}}}+\cdots.\label{pro}
%\end{split}
\end{gather}
Note that the $c_{\nu}$ are the same exact constants that appear in \eqref{pr}.

We shift our attention to the double series
\begin{align}\label{cc}
%\begin{split}
G_{\nu}(x,\sigma,\theta, s,w):=&\sum_{m=0}^{\infty}\sum_{n=0}^{\infty}
\Bigg(\frac{J_{\nu}(4\pi\sqrt{(m+\sigma)(n+\theta)x})}{(m+\sigma)^s(n+\theta)^w}\notag\\&\pm \frac{J_{\nu}(4\pi\sqrt{(m+1-\sigma)(n+\theta)x})}{(m+1-\sigma)^s(n+\theta)^w}
\pm\frac{J_{\nu}(4\pi\sqrt{(m+\sigma)(n+1-\theta)x})}{(m+\sigma)^s(n+1-\theta)^w}\notag\\
&\pm\frac{J_{\nu}(4\pi\sqrt{(m+1-\sigma)(n+1-\theta)x})}{(m+1-\sigma)^s(n+1-\theta)^w}\Bigg),
%\end{split}
\end{align}
where $x,\sigma,\theta,s,w$ are as before, and $\nu$ is an integer. The designation $\pm$ indicates either choice of sign, so that $G_{\nu}(x,\sigma,\theta, s,w)$ actually represents eight different double series. We need to consider these different combinations of signs because of the powers of $-1$ appearing as factors in \eqref{pro} and the analogous formulas for the other terms of \eqref{fun}. We will see below that the choice of signs does not affect the convergence of the series.

Theorem \ref{M} follows by combining Lemma \ref{L1} with the following result on the convergence of \eqref{cc}.
\begin{theorem}\label{M2}
Let $G_{\nu}(x,\sigma,\theta, s,w)$ be defined as above. Assume that $4\Realp(s)>1$, and that $4\Realp(w)>1$. Moreover, if $x$ is an integer, assume $\Realp(s)+\Realp(w)>\frac{25}{26}$, while if $x$ is not an integer, assume $\Realp(s)+\Realp(w)>\frac{5}{6}$. Then the double series $G_{\nu}(x,\sigma,\theta, s,w)$ converges uniformly with respect to $\sigma$ and $\theta$ in any compact subset of $(0,1)^2$.
\end{theorem}

The remaining part of this section is devoted to proving Theorem \ref{M2}.

We start by recalling the following fact. Let $z$ be a complex number, and $j$ a non-negative integer. Recall the notation for the binomial coefficient
%$$\left(\df{z}{j}\right) =\frac{z(z-1)(z-2)\cdots (z-j+1)}{j!}.$$
$${z\choose j} =\frac{z(z-1)(z-2)\cdots (z-j+1)}{j!}.$$
%\[{z\choose j} =\frac{z (z-1)(z-2)\cdots (z-j+1)}{j!}.\]
Then, for every complex number $\zeta$, with $|\zeta|<1$, we have the binomial theorem
\begin{equation*}\label{Nw}
(1+\zeta)^{z}=\sum_{j=0}^{\infty}{z\choose j} \zeta^j.
\end{equation*}
Using \eqref{Nw}, we can easily show that, for every $\theta\in (0,1)$,
%\begin{equation}
%\frac{1}{(n+\theta)^{z}}-\frac{1}{n^z}=\frac{1-(1+\frac{\theta}{n})^z}{(n+\theta)^z}=\frac{\frac{\theta}{n}+O(\frac{1}{n^2})}{(n+\theta)^z}
%\end{equation}
\begin{equation}\label{55}
\frac{1}{(n+\theta)^{z}}=\frac{1}{n^z}+O\bigg(\frac{1}{n^{z+1}}\bigg).
\end{equation}
This simple formula will be used several times in the following discussion.

Now recall the following asymptotic formulas, which hold for any positive integer $\nu$ \cite[p. 199]{watson}:
\begin{equation}\label{a1mar}
J_{\nu}(z)=\nu\bigg(\frac{2}{\pi z}\bigg)^{\frac{1}{2}}\cos\left(z-\frac{1}{2}\nu\pi-\frac{1}{4}\pi\right)+O\bigg(\frac{1}{z^{\frac{3}{2}}}\bigg),
\end{equation}
\begin{equation}\label{aa2}
J_{-\nu}(z)=\nu\bigg(\frac{2}{\pi z}\bigg)^{\frac{1}{2}}\cos\left(z+\frac{1}{2}\nu\pi-\frac{1}{4}\pi\right)+O\bigg(\frac{1}{z^{\frac{3}{2}}}\bigg).
\end{equation}
%Hence terms of the form \eqref{error} do not matter in proving convergence of the double series \eqref{fun}.
Let $\nu$ be a fixed non-zero integer, and let $\beta= -\frac{1}{2}\nu\pi-\frac{1}{4}\pi$. By \eqref{a1mar}, \eqref{aa2}, and \eqref{55}, to study the convergence of $G_{\nu}(x,\sigma,\theta, s,w)$, it is sufficient to investigate the convergence of
\begin{equation*}
\begin{split}
&S_1:=\sum_{m=0}^{\infty}\sum_{n=0}^{\infty}\Bigg{(}\frac{\cos(a\sqrt{(m+\sigma)(n+\theta)}+\beta)}
{(m+\sigma)^{s+\frac{1}{4}}(n+\theta)^{w+\frac{1}{4}}}
\pm\frac{\cos(a\sqrt{(m+\sigma)(n+1-\theta)}+\beta)}{(m+\sigma)^{s+\frac{1}{4}}(n+1-\theta)^{w+\frac{1}{4}}}\\
&\pm\frac{\cos(a\sqrt{(m+1-\sigma)(n+\theta)}+\beta)}{(m+1-\sigma)^{s+\frac{1}{4}}(n+\theta)^{w+\frac{1}{4}}}
\pm\frac{\cos(a\sqrt{(m+1-\sigma)(n+1-\theta)}+\beta)}{(m+1-\sigma)^{s+\frac{1}{4}}(n+1-\theta)^{w+\frac{1}{4}}}\Bigg).
\end{split}
\end{equation*}
Here for convenience we have set $a=4\pi\sqrt{x}$.

\subsection{Large values of $n$}
We examine the double sum $S_1$ for large values of $n$. We follow the arguments in \cite[p. 576--577]{Paper2}. Let $M$ and $N$ be integers with $M<N$. By the Euler-Maclaurin summation formula \cite[p. 619]{SF},
\begin{equation*}
\begin{split}
\sum_{n=M+1}^{N}\frac{\cos(a\sqrt{(m+\sigma)(n+\theta)}+\beta)}{(n+\theta)^{w+\frac{1}{4}}}
=&\int_{M+\theta}^{N+\theta} \frac{\cos(a\sqrt{(m+\sigma)t}+\beta)}{t^{w+\frac{1}{4}}}\, dt \\&+ \int_{M+\theta}^{N+\theta}\{t-\theta\}
\frac{d}{dt}\bigg(\frac{\cos(a\sqrt{(m+\sigma)t}+\beta)}{t^{w+\frac{1}{4}}}\bigg) dt,
\end{split}
\end{equation*}
where $\{t-\theta\}$ denotes the fractional part of $t-\theta$. Note that
\begin{equation*}
\begin{split}
\frac{d}{dt}\bigg(\frac{\cos(a\sqrt{(m+\sigma)t}+\beta)}{t^{w+\frac{1}{4}}}\bigg) =&\frac{1}{4t^{w+\frac{5}{4}}}\bigg(-2a\sqrt{(m+\sigma)t}\sin(a\sqrt{(m+\sigma)t}+\beta)\\
&-(4w+1)\cos(a\sqrt{(m+\sigma)t}+\beta)\bigg)\\ =&O\bigg(\frac{a\sqrt{m+\sigma}}{t^{\Realp w+\frac{3}{4}}}\bigg).
\end{split}
\end{equation*}
Hence,
\begin{equation}\label{fractional}
\int_{M+\theta}^{N+\theta}\{t-\theta\}\frac{d}{dt}\bigg(\frac{\cos(a\sqrt{(m+\sigma)t}+\beta)}{t^{w+\frac{1}{4}}}\bigg) dt=O\Bigg(\frac{a\sqrt{(m+\sigma)}}{(M+\theta)^{\Realp w-\frac{1}{4}}}\Bigg).
\end{equation}
Now let $u=a\sqrt{(m+\sigma)t}$. Then $t=\frac{u^2}{a^2(m+\sigma)}$ and $dt=\frac{2u}{a^2(m+\sigma)}du$. Thus,
\begin{equation}\label{aaMar}
\begin{split}
\int_{M+\theta}^{N+\theta} \frac{\cos(a\sqrt{(m+\sigma)t}+\beta)}{t^{w+\frac{1}{4}}}\, dt &=2(a^2(m+\sigma))^{w-\frac{3}{4}}
\int_{a\sqrt{(m+\sigma)(M+\theta)}}^{a\sqrt{(m+\sigma)(N+\theta)}}\frac{\cos(u+\beta)}{u^{2w-\frac{1}{2}}}\, du.\\
\end{split}
\end{equation}
Let $c,e$ be real constants. An integration by parts shows that
\begin{equation}\label{bb}
\begin{split}
\int_A^{B}\frac{c\sin u+e\cos u}{u^{2w-\frac{1}{2}}}\, du&=\frac{-c\cos u+e\sin u}{u^{2w-\frac{1}{2}}}\Bigg{\vert}^{A}_{B}+\left(2w-\df12\right)\int_A^B\frac{-c\cos u+e\sin u}{u^{2w+\frac{1}{2}}}\, du\\ &=O_w\Bigg(\frac{1}{A^{2\Realp w-\frac{1}{2}}}+\frac{1}{B^{2\Realp w-\frac{1}{2}}}\Bigg),
\end{split}
\end{equation}
as $A,B\to\infty$.
Recall the identity $\cos(u+\beta)=\cos u\cos\beta-\sin u\sin\beta$. By \eqref{aaMar} and \eqref{bb},
\begin{equation*}
\int_{M+\theta}^{N+\theta}\frac{\cos(a\sqrt{(m+\sigma)t}+\beta)}{t^{w+\frac{1}{4}}}\, dt = O\Bigg( \frac{1}{a\sqrt{m+\sigma}}\Bigg(\frac{1}{(M+\theta)^{\Realp w-\frac{1}{4}}}+\frac{1}{(N+\theta)^{\Realp w-\frac{1}{4}}}\Bigg)\Bigg).
\end{equation*}
Hence, by \eqref{fractional}, as $M\to\infty$,
\begin{equation*}
\begin{split}
\sum_{n=M}^{{\infty}}\frac{\cos(a\sqrt{(m+\sigma)(n+\theta)}+\beta)}{(n+\theta)^{w+\frac{1}{4}}}&=\lim_{N\to\infty}
\sum_{n=M}^{N}\frac{\cos(a\sqrt{(m+\sigma)(n+\theta)}+\beta)}{(n+\theta)^{w+\frac{1}{4}}}\\
&=O\Bigg(\frac{a\sqrt{m+\sigma}}{(M+\theta)^{\Realp w-\frac{1}{4}}}\Bigg).
\end{split}
\end{equation*}

Analogous results hold when $\sigma$ is replaced by $1-\sigma$ and when $\theta$ is replaced by $1-\theta$. We can thus let $M=[m^{1/(\Realp w-\frac{1}{4})}]$ and conclude that
\begin{equation*}
\begin{split}
&\sum_{n\geq m^{1/(\Realp w-\frac{1}{4})}}^{\infty}\Bigg{(}\frac{\cos(a\sqrt{(m+\sigma)(n+\theta)}+\beta)}
{(m+\sigma)^{s+\frac{1}{4}}(n+\theta)^{w+\frac{1}{4}}}
\pm\frac{\cos(a\sqrt{(m+\sigma)(n+1-\theta)}+\beta)}{(m+\sigma)^{s+\frac{1}{4}}(n+1-\theta)^{w+\frac{1}{4}}}\\
&\pm\frac{\cos(a\sqrt{(m+1-\sigma)(n+\theta)}+\beta)}{(m+1-\sigma)^{s+\frac{1}{4}}(n+\theta)^{w+\frac{1}{4}}}
\pm\frac{\cos(a\sqrt{(m+1-\sigma)(n+1-\theta)}+\beta)}{(m+1-\sigma)^{s+\frac{1}{4}}(n+1-\theta)^{w+\frac{1}{4}}}\Bigg)
\\&=O\Bigg(\frac{a}{m^{\Realp s+3/4}}\Bigg).
\end{split}
\end{equation*}
Recall that $\Realp s>\frac{1}{4}$. Hence, in our study of the uniform convergence of the sum $S_1$, we can replace it with the sum $S_2$, defined by
\begin{equation*}
\begin{split}
\sum_{m=0}^{\infty}&\sum_{0\leq n\leq m^{1/(\Realp w-\frac{1}{4})}}\Bigg{(}\frac{\cos(a\sqrt{(m+\sigma)(n+\theta)}
+\beta)}{(m+\sigma)^{s+\frac{1}{4}}(n+\theta)^{w+\frac{1}{4}}}
\pm\frac{\cos(a\sqrt{(m+\sigma)(n+1-\theta)}+\beta)}{(m+\sigma)^{s+\frac{1}{4}}(n+1-\theta)^{w+\frac{1}{4}}}\\
&\pm\frac{\cos(a\sqrt{(m+1-\sigma)(n+\theta)}+\beta)}{(m+1-\sigma)^{s+\frac{1}{4}}(n+\theta)^{w+\frac{1}{4}}}
\pm\frac{\cos(a\sqrt{(m+1-\sigma)(n+1-\theta)}+\beta)}{(m+1-\sigma)^{s+\frac{1}{4}}(n+1-\theta)^{w+\frac{1}{4}}}\Bigg).
\end{split}
\end{equation*}

\subsection{Small values of $n$} Let $\delta>0$ be a small positive number to be specified later, and let $S_3$ be the same double series as $S_2$ but with the sum on $n$ performed over the interval $0\leq n\leq m^{1-\delta}$. We now prove convergence of $S_3$, using the following result from \cite{bessie3}.
\begin{lemma}\cite[p. 31--33]{bessie3}\label{Le}
Consider the sum
%\[ S(\alpha,\beta,\mu,H_1, H_2)=\sum_{H_1< m\leq H_2} \frac{\cos (\alpha\sqrt{m+\mu}+\beta)}{(m+\mu)^{s+\frac{1}{4}}}\]
$$ S(\alpha,\beta,\mu,H_1, H_2)=\sum_{H_1< m\leq H_2} \frac{\cos (\alpha\sqrt{m+\mu}+\beta)}{(m+\mu)^{s+\frac{1}{4}}},$$
where $\alpha>0,\beta\in\mathbb{R},\mu\in [0,1]$, and $H_1<H_2$, where $H_1$ and $H_2$ are large positive integers. Assume also that
%\[c_1\leq \alpha\leq c_2 H_1^{(1-\delta)/2},\]
$$c_1\leq \alpha\leq c_2 H_1^{(1-\delta)/2},$$
where $c_1$ and $c_2$ are positive constants and $\delta>0$ is a fixed small positive real number. Then
%\[S(\alpha,\beta,\mu, H_1, H_2)=O\Bigg(\frac{1}{\alpha H_1^{\Realp s-\frac{1}{4}}}\Bigg).\]
$$S(\alpha,\beta,\mu, H_1, H_2)=O\Bigg(\frac{1}{\alpha H_1^{\Realp s-\frac{1}{4}}}\Bigg).$$
\end{lemma}

We write $S_{3,M}$ for the partial sum in $S_3$, where the summation over $m$ is restricted to $1\leq m\leq M$. To prove the convergence of $S_3$, we use Cauchy's criterion. That is, for every $\epsilon>0$, we show that there exists $M_{\epsilon}$ such that $|S_{3,M_2}-S_{3,M_1}|<\epsilon$ whenever $M_1,M_2>M_{\epsilon}$. Exchanging the order of summation in a term of $S_3$ yields %We focus on the two terms only justified since we prove converhence. For large $M_1,M_2$ we exchange the order of summation
\begin{equation*}
\begin{split}
&\sum_{m=M_1}^{M_2}\sum_{0\leq n< m^{1-\delta}}\frac{\cos(a\sqrt{(m+\sigma)(n+\theta)}+\beta)}{(m+\sigma)^{s+\frac{1}{4}}(n+\theta)^{w+\frac{1}{4}}}
\\&=\sum_{0\leq n\leq M_2^{1-\delta}}\sum_{\max\{n^{1/(1-\delta)},M_1\}<m\leq M_2}\frac{\cos(a\sqrt{(m+\sigma)(n+\theta)}+\beta)}{(m+\sigma)^{s+\frac{1}{4}}(n+\theta)^{w+\frac{1}{4}}}.
\end{split}
\end{equation*}
We now apply Lemma \ref{Le} with $\alpha=a\sqrt{n+\theta}, \mu=\sigma, H_1=\max\{n^{1/(1-\delta)},M_1\}$, and $H_2=M_2$. Hence, %since $\text{Re } w>\tf14$,
\begin{equation*}
\begin{split}
&\sum_{m=M_1}^{M_2}\sum_{0\leq n< m^{1-\delta}}\frac{\cos(a\sqrt{(m+\sigma)(n+\theta)}+\beta)}{(m+\sigma)^{s+\frac{1}{4}}(n+\theta)^{w+\frac{1}{4}}}\\&= O_{a,\delta}\Bigg(\sum_{0\leq n\leq M_2^{1-\delta}}\frac{1}{(n+\theta)^{w+\frac{3}{4}}\max\{n^{1/(1-\delta)},M_1\}^{\Realp s-\frac{1}{4}}}\Bigg)\\&= O_{a,\delta}\Bigg(\sum_{0\leq n\leq M_1^{1-\delta}}\frac{1}{(n+\theta)^{w+\frac{3}{4}}M_1^{\Realp s-\frac{1}{4}}}\Bigg)\\&+ O_{a,\delta}\Bigg(\sum_{M_1^{1-\delta}\leq n\leq M_2^{1-\delta}}\frac{1}{(n+\theta)^{w+\frac{3}{4}+(\Realp s-\frac{1}{4})/(1-\delta)}}\Bigg)\\&=O_{a,\delta,w}\Bigg(\frac{1}{M_1^{\Realp s-\frac{1}{4}}}\Bigg),
\end{split}
\end{equation*}
where in the last step we used \eqref{55} and the hypothesis $\Realp w>\frac{1}{4}$.
The same reasoning applies to every term of the sum $S_3$, by replacing $\theta$ with $1-\theta$ and $\sigma$ with $1-\sigma$. Since $\Realp s>\frac{1}{4}$ by hypothesis, we conclude that $|S_{3,M_2}-S_{3,M_1}|<\epsilon$ when $M_1$ is large enough. We have thus proved the convergence of $S_3$, uniformly for $(\sigma,\theta)$ in a compact subset of $(0,1)^2$. Note that in our situation we need the full strength of Lemma \ref{Le}, while the corresponding reduction in \cite[p. 33]{bessie3} only requires the case $\mu=0$.

\subsection{Further reductions}
Assume that $\delta<1/(\Realp w-\frac{1}{4})-1$. Let $S_4$ be the same double series as $S_2$ but with the sum on $n$ performed over the interval $m^{1+\delta} < n\leq m^{1/(\Realp w-\frac{1}{4})}$. The techniques used above to prove the convergence of $S_3$ also show the convergence of $S_4$. Indeed, Lemma \ref{Le}, applied with $\alpha=a\sqrt{n+\theta}, \mu=\theta, H_1=m^{1+\delta}$, and $H_2=m^{1/(\Realp w-\frac{1}{4})}$ yields
\begin{equation}\label{rt}
\begin{split}
&\sum_{m=M_1}^{M_2}\sum_{m^{1+\delta}< n\leq m^{1/(\Realp w-\frac{1}{4})}}\frac{\cos(a\sqrt{(m+\sigma)(n+\theta)}+\beta)}{(m+\sigma)^{s+\frac{1}{4}}(n+\theta)^{w+\frac{1}{4}}}\\
&=O_{a,\delta}\Bigg(\sum_{m=M_1}^{M_2}\frac{1}{(m+\sigma)^{s+\frac{3}{4}}m^{(1+\delta)/(\Realp w-\frac{1}{4})}}\Bigg)\\&=O_{a,\delta,s}\Bigg(\frac{1}{M_1^{(1+\delta)/(\Realp w-\frac{1}{4})}}\Bigg),
\end{split}
\end{equation}
where the second equality follows from $\Realp s>\frac{1}{4}$. Since $\Realp w>\frac{1}{4}$, the upper bound \eqref{rt} can be used to apply Cauchy's criterion to the partial sums of $S_4$, thus showing convergence. Given our analysis of $S_3$ and $S_4$, we can replace the series  $S_2$ in our study of  convergence by the sum $S_5$, defined by
\begin{equation*}
\begin{split}
\sum_{m=0}^{\infty}&\sum_{m^{1-\delta}< n\leq m^{1+\delta}}\Bigg{(}\frac{\cos(a\sqrt{(m+\sigma)(n+\theta)}+\beta)}{(m+\sigma)^{s+\frac{1}{4}}(n+\theta)^{w+\frac{1}{4}}}
\pm\frac{\cos(a\sqrt{(m+\sigma)(n+1-\theta)}+\beta)}{(m+\sigma)^{s+\frac{1}{4}}(n+1-\theta)^{w+\frac{1}{4}}}\\
&\pm\frac{\cos(a\sqrt{(m+1-\sigma)(n+\theta)}+\beta)}{(m+1-\sigma)^{s+\frac{1}{4}}(n+\theta)^{w+\frac{1}{4}}}
\pm\frac{\cos(a\sqrt{(m+1-\sigma)(n+1-\theta)}+\beta)}{(m+1-\sigma)^{s+\frac{1}{4}}(n+1-\theta)^{w+\frac{1}{4}}}\Bigg).
\end{split}
\end{equation*}
By \eqref{55}, it is sufficient to prove the convergence of $S_5$ with $(n+\theta)^{w+\frac{1}{4}}$ and $(n+1-\theta)^{w+\frac{1}{4}}$ replaced by $n^{w+\frac{1}{4}}$. Analogously, we can replace $(m+\sigma)^{s+\frac{1}{4}}$ and $(m+1-\sigma)^{s+\frac{1}{4}}$ by $m^{s+\frac{1}{4}}$. Recall the identities
\begin{equation*}
\cos a+\cos b= 2\cos\bigg(\frac{a+b}{2}\bigg)\cos\bigg(\frac{a-b}{2}\bigg),
\end{equation*}
\begin{equation*}
\cos a -\cos b= 2\sin\bigg(\frac{a+b}{2}\bigg)\sin\bigg(\frac{a-b}{2}\bigg),
\end{equation*}
and the asymptotic formulas
\begin{equation*}
\sqrt{n+\theta}-\sqrt{n+1-\theta}=\frac{2\theta-1}{2\sqrt{n}}+O\bigg(\frac{1}{n^{\frac{3}{2}}}\bigg),
\end{equation*}
\begin{equation*}
\frac{\sqrt{n+\theta}+\sqrt{n+1-\theta}}{2}=\sqrt{n+\frac{1}{2}}+O\bigg(\frac{1}{n^{\frac{3}{2}}}\bigg).
\end{equation*}
Hence the convergence of $S_5$ will be proved if we can show, under the hypotheses of our theorem, the convergence of the double series
\begin{equation}\label{3una}
\sum_{m=0}^{\infty}\sum_{m^{1-\delta}<n\leq m^{1+\delta}}\frac{\cos\bigg(a\sqrt{(m+\sigma)(n+\frac{1}{2})}
+\beta\bigg)\cos\bigg(\frac{a(2\theta-1)}{4}\sqrt{\frac{m+\sigma}{n}}\bigg)}{m^{s+\frac{1}{4}}n^{w+\frac{1}{4}}},
\end{equation}
as well as the convergence of \eqref{3una} with both occurrences of the function $\cos$ replaced by $\sin$. Now see \cite{Paper2} from Section 4.3 onwards and use the same arguments to complete the proof.

\section{Balanced identities}
Having established the convergence result in Theorem \ref{M}, we are now ready to prove Theorem \ref{BI}.

 \begin{proof}[Proof of Theorem \ref{BI}]
When $k=0$, Equation \eqref{ii} becomes
 \begin{equation}\label{iii}
 \begin{split}
&{\sum_{mn\leq x}}^{\prime}\cos(2\pi m\sigma)\sin(2\pi n\theta)+\frac{\cot(\pi\theta)}{4}\\=&\frac{\sqrt{x}}{4}\sum_{m=0}^{\infty}\sum_{n=0}^{\infty}
\Bigg{\{} \frac{J_1(4\pi\sqrt{(m+\sigma)(n+\theta)x})}{\sqrt{(m+\sigma)(n+\theta)}}
+\frac{J_1(4\pi\sqrt{(m+1-\sigma)(n+\theta)x})}{\sqrt{(m+1-\sigma)(n+\theta)}}\\
&-\frac{J_1(4\pi\sqrt{(m+\sigma)(n+1-\theta)x})}{\sqrt{(m+\sigma)(n+1-\theta)}}
-\frac{J_1(4\pi\sqrt{(m+1-\sigma)(n+1-\theta)x})}{\sqrt{(m+1-\sigma)(n+1-\theta)}}
 \Bigg{\}}.
 \end{split}
 \end{equation}
 Note that convergence of the double sum on the right-hand side of \eqref{iii} is a consequence of Theorem \ref{M}. To prove \eqref{iii} it is therefore sufficient to compute the Fourier coefficients of both sides of the equation and show that they are equal.

 The general statement of Theorem \ref{BI} follows from the case $k=0$ by taking derivatives on both sides of \eqref{iii}. Such derivatives can be brought inside the infinite sums because the double series is uniformly convergent by Theorem \ref{M} applied in the special case $s=w=\frac{1}{2}$.

 In \cite[Theorems 4.1, 4.4]{bessie4}, three of the present authors first  proved \eqref{iii} by showing that the Fourier sine series of both sides are identical.  Secondly, they proved \eqref{iii}, but with the order of summation reversed, by demonstrating that the Fourier cosine series of both sides are identical.  Thus, it was shown that one could reverse the order of summation by proving that the two iterated sums converge to the same limit.  In our analysis above, using uniform convergence, we also demonstrated that the two iterated series converged to the same limit. But now appealing to the aforementioned two theorems in \cite{bessie4}, we have completed the proof of \eqref{iii}, and consequently of Theorem \ref{BI}.
\end{proof}

Theorem \ref{TI} and Theorem \ref{TT} can be proved similarly to Theorem \ref{BI}. We describe below the necessary modifications to the arguments presented above.

\begin{proof}[Proof of Theorems \ref{TI} and \ref{TT}]
In the case $k=0$ the two theorems yield known identities (see \cite[Theorem 2.1]{bessel2bbskaz} and \cite[Theorem 2.3]{bessel2bbskaz}, respectively). In \cite{bessel2bbskaz} these identities were proved with the iterated sums replaced by double sums where, for brevity, the products of the indices $m$ and $n$ tend to infinity.  However, in each case, the  same arguments developed there can be used to prove the identities with  iterated sums, yielding \eqref{funI} and \eqref{funT} in the case $k=0$. Theorems \ref{TI} and \ref{TT} then follow from the case $k=0$ by taking derivatives on both sides of the identities.

To conclude the proofs of Theorems \ref{TI} and \ref{TT} we need theorems on the uniform convergence of the two double series appearing on the right sides of \eqref{funI} and \eqref{funT}. To this end, we prove that Theorem \ref{M} holds even when, in the definition \eqref{fun} of $G^{\alpha,\beta}(x,\sigma,\theta,s,w)$, the expression inside the brackets is replaced with a double series involving the function $I_1$ (respectively $T_{\frac{3}{2}}$) corresponding to the one on the right-hand side of \eqref{funI} (respectively \eqref{funT}). We denote these two new series by $G_I^{\alpha,\beta}(x,\sigma,\theta,s,w)$ and $G_T^{\alpha,\beta}(x,\sigma,\theta,s,w)$, respectively.

 In the case of $G_I^{\alpha,\beta}(x,\sigma,\theta,s,w)$ the proof of Theorem \ref{M} carries over with minor modifications. One starts from an analogue of Lemma \ref{L1}, which holds if one simply replaces in the statement every occurrence of $J_{\nu}$ with $I_{\nu}$. In the proof, instead of Equation \eqref{mara}, one uses the following similar recurrence relation \cite[p. 79]{watson}, which holds for every integer $\nu$: \begin{equation}\label{maraI}
I_{\nu-1}(z)+I_{\nu+1}(z)=2I'_{\nu}(z).
\end{equation}
Later in the proof, the asymptotic formulas \eqref{a1mar} and \eqref{aa2}, which hold for every positive integer $\nu$, are replaced by
\begin{equation}\label{a1marI}
I_{\nu}(z)=-\nu\bigg(\frac{2}{\pi z}\bigg)^{\frac{1}{2}}\sin\left(z-\frac{1}{2}\nu\pi-\frac{1}{4}\pi\right)+O\bigg(\frac{1}{z^{\frac{3}{2}}}\bigg),
\end{equation}
\begin{equation}\label{aaI}
I_{-\nu}(z)=\nu\bigg(\frac{2}{\pi z}\bigg)^{\frac{1}{2}}\sin\left(z+\frac{1}{2}\nu\pi-\frac{1}{4}\pi\right)+O\bigg(\frac{1}{z^{\frac{3}{2}}}\bigg).
\end{equation}
 These asymptotic formulas arise from combining the definition \eqref{defI} with the asymptotic expressions for $Y_{\nu}$ and $K_{\nu}$ given in \cite[p. 199]{watson} and \cite[p. 202]{watson}, respectively. One can now follow the proof of Theorem \ref{M} with the appropriate minor changes due to the fact that the trigonometric function $\cos$ has been replaced with $\sin$. It can be readily checked that this change does not affect the convergence of the corresponding double series.

 Let us now consider $G_T^{\alpha,\beta}(x,\sigma,\theta,s,w)$. In order to prove the convergence of this double series we need the following analogue of Lemma \ref{L1}.

\begin{lemma}\label{L1bis}
The following identity holds:
\begin{equation*}\label{prbis}
\frac{\partial^{\alpha+\beta}}{\partial\sigma^{\alpha}\partial\theta^{\beta}} \frac{T_{\frac{3}{2}}(4\pi^2(m+\sigma)(n+\theta)x)}{(m+\sigma)^s(n+\theta)^w}=\frac{\sum_{\nu\in A}c_{\nu} Q_{\nu}(4\pi\sqrt{(m+\sigma)(n+\theta)x})}{(m+\sigma)^{s+\frac{\alpha}{2}-\frac{\beta}{2}}(n+\theta)^{w+\frac{\beta}{2}-\frac{\alpha}{2}}}+\cdots.
\end{equation*}
Here $A$ is a finite subset of $\mathbb{Z}$, the $c_{\nu}$ are constants, and the dots $\cdots$ denote a sum of terms that have the form
\begin{equation*}\label{pbis}  r_{\nu}\,\frac{Q_{\nu}(4\pi\sqrt{(m+\sigma)(n+\theta)x})}{(m+\sigma)^{\gamma}(n+\theta)^{\delta}},\quad \nu\in\mathbb{Z},\, r_{\nu}\in\mathbb{R},
\end{equation*}
where
\begin{equation*}\label{obis}
\begin{cases}
\Realp(\gamma)\geq \Realp(s)+\frac{\alpha}{2}-\frac{\beta}{2},\\
\Realp(\delta)\geq \Realp(w)+\frac{\beta}{2}-\frac{\alpha}{2},\\
\end{cases}
\end{equation*}
with at least one of the two inequalities in \eqref{obis} being strict. With $Q_{\nu}$ we denote one of the Bessel functions $ I_{\nu}, K_{\nu}, Y_{\nu}$.
\end{lemma}
The first step in the proof of this lemma is to recall a formula for $T_{\frac{3}{2}}(y^2)$ as a linear combination of Bessel functions.  From \cite[p.~90, Equation (5.9)]{bessel2bbskaz}, if $y=2\pi\sqrt{nx/pq}$, where $p$ and $q$ are primes, then
\begin{equation*}\label{T32}
T_{\frac{3}{2}}(y^2)=\df{1}{2y^2}I_1(2y)-\df{1}{y}I_1^{\prime}(2y)+Y_1(2y)-\df{2}{\pi}K_1(2y).
\end{equation*}
Lemma \ref{L1bis} can now be proved by induction, similarly to Lemma \ref{L1}. One needs to use, instead of \eqref{mara}, the corresponding recurrence  relations  for the functions $I_{\nu}, K_{\nu}, Y_{\nu}$ \cite[p. 79 formula (2); p. 66 formula (2)]{watson}. To finish the proof of the convergence of $G_T^{\alpha,\beta}(x,\sigma,\theta,s,w)$, one can now follow the same steps as in the proof of the convergence of \eqref{cc}, using the asymptotic expansions given in \eqref{a1marI}, \eqref{aaI}, \cite[p. 199]{watson}, and \cite[p. 202]{watson}.
\end{proof}

\section{Application: a balanced identity of order 1}\label{last}
We record the formulas for $k=0$ from Theorems \ref{TI} and \ref{TT}.  As indicated above, they were proved in \cite[pp.~70--71, Theorems 2.1, 2.3]{bessel2bbskaz} under hypotheses that were stronger than necessary.

\begin{theorem}\label{even'}
Let $I_1(x)$ be defined by \eqref{defI}. If $0<\theta,$ $\sigma <1$
and $x>0,$ then
\begin{align}\label{ccc}
&{\sum_{nm\leq x}}^{\prime}\cos(2\pi n\theta)\cos(2\pi m\sigma) \\
=&\frac14+\frac{\sqrt{x}}{4}\sum_{n,m\geq
0}\left\{\frac{I_1(4\pi\sqrt{(n+\theta)(m+\sigma)x})}{\sqrt{(n+\theta)(m+\sigma)}}
+\frac{I_1(4\pi\sqrt{(n+1-\theta)(m+\sigma)x})}{\sqrt{(n+1-\theta)(m+\sigma)}}\right.\nonumber
\\
&\qquad \qquad +\left.
\frac{I_1(4\pi\sqrt{(n+\theta)(m+1-\sigma)x})}{\sqrt{(n+\theta)(m+1-\sigma)}}+
\frac{I_1(4\pi\sqrt{(n+1-\theta)(m+1-\sigma)x})}{\sqrt{(n+1-\theta)(m+1-\sigma)}}\right\}.\nonumber
\end{align}
\end{theorem}

\begin{theorem}\label{sss}
If $0<\theta,$ $\sigma <1$ and $x>0,$ then
\begin{align}\label{sisi}
&{\sum_{nm\leq x}}^{\prime}nm\sin(2\pi n\theta)\sin(2\pi m\sigma) \\
=&\frac{x\sqrt{x}}{4}\sum_{n,m\geq
0}\left\{\frac{T_{\frac32}\big(4\pi^2(n+\theta)(m+\sigma)x\big)}{\sqrt{(n+\theta)(m+\sigma)}}
-\frac{T_{\frac32}\big(4\pi^2(n+1-\theta)(m+\sigma)x\big)}{\sqrt{(n+1-\theta)(m+\sigma)}}\right.\nonumber
\\
&\qquad \qquad -\left.
\frac{T_{\frac32}\big(4\pi^2(n+\theta)(m+1-\sigma)x\big)}{\sqrt{(n+\theta)(m+1-\sigma)}}+
\frac{T_{\frac32}\big(4\pi^2(n+1-\theta)(m+1-\sigma)x\big)}{\sqrt{(n+1-\theta)(m+1-\sigma)}}\right\}.
\nonumber
\end{align}
where $T_{3/2} $ is defined in \eqref{t32}.
\end{theorem}

From our remarks above on the uniform convergence of the right-hand side of \eqref{funI}, differentiating the identity \eqref{funI} for $k=1$ yields an identity for the left-hand side of \eqref{funT} for $k=0$.  However, we only know the precise nature of the right-hand side of \eqref{funT} because of the independent proof of \eqref{funT} in \cite{bessel2bbskaz}, but under stronger  hypotheses, as emphasized above.
Our goal is to show directly that the first mixed partial derivative of \eqref{ccc}  is equal to \eqref{sisi}.  The needed termwise differentiation is justified by Theorem \ref{TI}.

Let
\begin{equation}\label{aII}
u=4\pi\sqrt{(n+\theta)(m+\sigma)x}.
\end{equation}
Then,
\begin{align*}
\df{\partial}{\partial\theta}\left\{\df{I_1(u)}{\sqrt{(n+\theta)(m+\sigma)}}\right\}=&
\df{I^{\prime}(u)2\pi\sqrt{x}}{(n+\theta)}-\df{I_1(u)}{2(n+\theta)^{3/2}(m+\sigma)^{1/2}}
\end{align*}
and
\begin{align}
&\df{\partial^2}{\partial\sigma\partial\theta}\left\{\df{I_1(u)}{\sqrt{(n+\theta)(m+\sigma)}}\right\}=
\df{I^{\prime\prime}(u)2\pi\sqrt{x}}{(n+\theta)} 2\pi\sqrt{\df{(n+\theta)x}{(m+\sigma)}}\notag\\
&-\df{I_1^{\prime}(u)}{2(n+\theta)^{3/2}(m+\sigma)^{1/2}}2\pi\sqrt{\df{(n+\theta)x}{(m+\sigma)}}
+\df{I_1(u)}{4(n+\theta)^{3/2}(m+\sigma)^{3/2}}\notag\\
=&\df{4\pi^2 x I_1^{\prime\prime}(u)}{\sqrt{(n+\theta)(m+\sigma)}}
-\df{\pi\sqrt{x}I_1^{\prime}(u)}{(n+\theta)(m+\sigma)}+\df{I_1(u)}{4(n+\theta)^{3/2}(m+\sigma)^{3/2}}.\label{a22}
\end{align}

We now use \cite[p.~66, formula (3); p.~79, formula (3)]{watson}, respectively,
\begin{equation}
uY_1^{\prime}(u)+Y_1(u)=uY_0(u)\qquad\text{and}\qquad
uK_1^{\prime}(u)+K_1(u)=-uK_0(u).\label{YK1}
\end{equation}
Also, \cite[p.~66, formula (4); p.~79, formula (4)]{watson}, respectively,
\begin{align}\label{YK}
Y_0^{\prime}(u)=-Y_1(u) \qquad \text{and} \qquad K_0^{\prime}(u)=-K_1(u).
\end{align}
Thus, from \eqref{defI}, \eqref{YK1}, and \eqref{YK},
\begin{align}
I_{1}^{'}(u)&=-Y_{1}^{'}(u)-\df{2}{\pi}K_{1}^{'}(u)\notag\\
&=\df{1}{u}Y_1(u)-Y_0(u)+\df{2}{\pi u}K_1(u)+\df{2}{\pi}K_0(u),\label{YK2}
\end{align}
and from \eqref{YK} and \eqref{YK2},
\begin{align}\label{a3}
I_1^{\prime\prime}(u)=&\df{1}{u}Y_1^{\prime}(u)-\df{1}{u^2}Y_1(u)-Y_0^{'}(u)+\df{2}{\pi u}K_1^{'}(u)-\df{2}{\pi u^2}K_1(u)+\df{2}{\pi}K_0^{'}(u)\notag\\
=&\df{1}{u}\left(-\df{1}{u}Y_1(u)+Y_0(u)\right)-\df{1}{u^2}Y_1(u)+Y_1(u)+\df{2}{\pi u}\left(-\df{1}{u}K_1(u)-K_0(u)\right)\notag\\&-\df{2}{\pi u^2}K_1(u)-\df{2}{\pi}K_1(u)\notag\\
=&\df{2}{u^2}\left(-Y_1(u)-\df{2}{\pi}K_1(u)\right)+\df{1}{u}\left(Y_0(u)-\df{2}{\pi}K_0(u)\right)+Y_1(u)-\df{2}{\pi}K_1(u).
\end{align}

Now return to \eqref{a22} and substitute from \eqref{aII}, \eqref{defI}, \eqref{YK2}, and \eqref{a3} to deduce that
\begin{align*}%\label{a4}
&\df{\partial^2}{\partial\sigma\partial\theta}\left\{\df{I_1(u)}{\sqrt{(n+\theta)(m+\sigma)}}\right\}\notag\\
=&\df{16\pi^3 x^{3/2}}{u}\left\{-\df{2}{u^2}Y_1(u)-\df{4}{\pi u^2}K_1(u)+\df{1}{u}Y_0(u)-\df{2}{\pi u}K_0(u)+Y_1(u)-\df{2}{\pi}K_1(u)\right\}\notag\\
&-\df{16\pi^3x^{3/2}}{u^2}\left\{\df{1}{u}Y_1(u)-Y_0(u)+\df{2}{\pi u}K_1(u)+\df{2}{\pi}K_0(u)\right\}\notag\\
&-\df{16\pi^3x^{3/2}}{u^3}\left\{Y_1(u)+\df{2}{\pi }K_1(u)\right\}\notag\\
=&x^{3/2}\left\{-\df{64\pi^3}{u^3}Y_1(u)-\df{128\pi^2}{u^3}K_1(u)+\df{16\pi^3}{u}Y_1(u)
-\df{32\pi^2}{u}K_1(u)\right.\notag\\&+\left.\df{32\pi^3}{u^2}Y_0(u)-\df{64\pi^2}{u^2}K_0(u)\right\},
\end{align*}
or
\begin{align}\label{a5}
&\df{1}{4\pi^2}\df{\partial^2}{\partial\sigma\partial\theta}\left\{\df{I_1(u)}{\sqrt{(n+\theta)(m+\sigma)}}\right\}\notag\\
=&x^{3/2}\left\{-\df{16\pi}{u^3}Y_1(u)-\df{32}{u^3}K_1(u)+\df{4\pi}{u}Y_1(u)
-\df{8}{u}K_1(u)+\df{8\pi}{u^2}Y_0(u)-\df{16}{u^2}K_0(u)\right\}.
\end{align}

We now turn to $T_{3/2}(z)$.
If we set $u=2y$ in \eqref{T32}, we find that
\begin{equation}\label{a6}
T_{3/2}(\tf14 u^2)=\df{2}{u^2}I_1(u)-\df{2}{u}I_1^{\prime}(u)+Y_1(u)-\df{2}{\pi}K_1(u).
\end{equation}
Appealing to \eqref{YK2}, we see from \eqref{a6} that
\begin{align}\label{a77}
T_{3/2}(\tf14 u^2)=&-\df{2}{u^2}Y_1(u)-\df{4}{\pi u^2}K_1(u)-\df{2}{u^2}Y_1(u)+\df{2}{u}Y_0(u)-\df{4}{\pi u^2}K_1(u)-\df{4}{\pi u}K_0(u)\notag\\&+Y_1(u)-\df{2}{\pi}K_1(u)\notag\\
=& -\df{4}{u^2}Y_1(u)-\df{8}{\pi u^2}K_1(u)+\df{2}{u}Y_0(u)-\df{4}{\pi u}K_0(u)+Y_1(u)-\df{2}{\pi}K_1(u).
\end{align}
Recalling the definition \eqref{aII}, we can write \eqref{a77} in the form
\begin{align}\label{a7}
&\df{T_{3/2}(\tf14 u^2)}{\sqrt{(n+\theta)(m+\sigma)}}=\df{4\pi\sqrt{x}T_{3/2}(\tf14 u^2)}{u}\notag\\
=&\sqrt{x}\left(-\df{16\pi}{u^3}Y_1(u)-\df{32}{u^3}K_1(u)+\df{8\pi}{u^2}Y_0(u)-\df{16}{u^2}K_0(u)+\df{4\pi}{u}Y_1(u)-\df{8}{u}K_1(u)\right).
\end{align}

Hence, from \eqref{a5} and \eqref{a7}, we conclude that the first balanced derivative of the terms in the first series on the right side of \eqref{ccc}, multiplied by $1/(4\pi^2)$, are equal to the terms of the first series on the right-hand side of \eqref{sisi}.  If we successively set
\begin{align*} u=&4\pi\sqrt{(n+1-\theta)(m+\sigma)x},\\  u=&4\pi\sqrt{(n+\theta)(m+1-\sigma)x},\\  u=&4\pi\sqrt{(n+1-\theta)(m+1-\sigma)x},
\end{align*}
we can make analogous conclusions for the second, third, and fourth series terms on the right-hand sides of \eqref{ccc} and \eqref{sisi}.  In conclusion, we have shown that taking the first balanced derivative of \eqref{ccc} yields  \eqref{sisi}, as expected.

\section{General Theorems of Chandrasekharan and Narasimhan}

 We offer two general theorems of Chandrasekharan and Narasimhan \cite{annals2}, which we employ in the sequel.  First, we
 provide the general setting \cite[p.~93--96]{annals2}.

 \begin{definition}\label{def} Let $a(n)$ and $b(n)$ be two sequences of complex numbers, where not all terms are equal to 0 in either sequence. Let $\lambda_n$ and $\mu_n$ be two sequences of positive numbers, strictly increasing to $\infty$. Let $\delta>0$. Throughout, $s=\sigma +it$, where $\sigma$ and $t$ are both real. Let
 \begin{equation}\label{Delta}
 \Delta(s):=\prod_{n=1}^{N}\Gamma(\alpha_n s+\beta_n),
 \end{equation}
 where $N\geq 1$, $\beta_n, 1\leq n\leq N$, is any complex number, and $\alpha_n>0, 1\leq n \leq N$.  Assume that
 \begin{equation*}\label{A}
 A:=\sum_{n=1}^N\alpha_n\geq 1.
 \end{equation*}
 Let
 \begin{equation*}\label{phipsi}
 \varphi(s):=\sum_{n=1}^{\infty}\df{a(n)}{\lambda_n^s}\qquad \text{and}\qquad \psi(s):=\sum_{n=1}^{\infty}\df{b(n)}{\mu_n^s}
 \end{equation*}
  converge absolutely in some half-plane, and suppose they satisfy the functional equation
 \begin{equation}\label{fe}
 \Delta(s)\varphi(s)=\Delta(\delta-s)\psi(\delta-s).
 \end{equation}
 Furthermore, assume that
   there exists in the $s$-plane a domain $\mathfrak{D}$, which is the exterior of a compact set $S$, in which there exists an analytic function $\chi$ with the properties
 \begin{equation*}
 \lim_{|t|\to\infty}\chi(s)=0,
 \end{equation*}
 uniformly in every interval $-\infty<\sigma_1\leq \sigma\leq\sigma_2<\infty$, and
 \begin{align*}
 \chi(s)&=\Delta(s)\varphi(s), \qquad \sigma > \alpha,\\
 \chi(s)&=\Delta(\delta-s)\psi(\delta-s), \qquad \sigma<\beta,
 \end{align*}
 where $\alpha$ and $\beta$ are particular constants.
 \end{definition}

For $\rho\geq0$, let
\begin{equation}\label{omega6}
A_{\rho}(x):=\df{1}{\Gamma(\rho+1)}\sum_{\lambda_n\leq x}a(n)(x-\lambda_n)^{\rho},
\end{equation}
where the prime $\prime$ indicates that if $x=\lambda_n$ and $\rho=0$, the last term is to be multiplied by $\tf12$. Furthermore, let
 \begin{equation}\label{poles}
 Q_{\rho}(x):=\df{1}{2\pi i}\int_{\mathcal{C}}\df{\Gamma(s)\varphi(s)}{\Gamma(s+\rho+1)}x^{s+\rho}ds,
 \end{equation}
 where $\mathcal{C_{\rho}}$ is a closed curve enclosing all of the singularities of the integrand to the right of $\sigma=-\rho-1-k$, where $k$ is chosen such that $k>|\delta/2-1/(4A)|$, and all of the singularities of $\varphi(s)$ lie in $\sigma>-k$.
  %Chandrasekharan and Narasimhan \cite[p.~98]{annals2} impose restrictions on the singularities of $\varphi(s)$.  However, in each of our applications, $\varphi(s)$ is an entire function.
   In Sections \ref{first}--\ref{third}, $\rho=0$, and so in these sections we write $A_{\rho}(x)=A(x)$ and $Q_{\rho}(x)=Q(x)$.  In Section \ref{fourth}, we consider the more general case when $\rho>0$.

 \begin{theorem}\label{omega1}\cite[p.~98, Theorem 3.1]{annals2}
 Suppose that $\varphi(s)$ and $\psi(s)$ satisfy \eqref{fe}.  Suppose that $\left\{\mu_n\right\}$ contains a subset $\left\{\mu_{n_k}\right\}$ such that no number $\mu_n^{1/(2A)}$ is represented as a linear combination of the numbers $\left\{\mu_{n_k}^{1/(2A)}\right\}$ with the coefficients $\pm1$, unless $\mu_n^{1/(2A)}=\mu_{n_r}^{1/(2A)}$ for some $r$, in which case $\mu_n^{1/(2A)}$ has no other representation. Suppose furthermore that
 \begin{equation}\label{omega2}
 \sum_{n=1}^{\infty}\df{|\Realp\,b(n_k)|}{\mu_{n_k}^{(A\delta+\rho+1/2)/(2A)}}=+\infty.
 \end{equation}

 Set
 \begin{equation}\label{omega3}
 \theta:=\df{A\delta+\rho(2A-1)-\tf12}{2A}.
 \end{equation}
% If
% \begin{equation}\label{omega6}
% A(x):={\sum_{\lambda_n\leq x}}^{\prime}a(n),
% \end{equation}
 Then
 \begin{align}
 \varlimsup_{x\to\infty}\df{\Realp\{A_{\rho}(x)-Q_{\rho}(x)\}}{x^{\theta}}&=+\infty, \label{omega4}\\
 \varliminf_{x\to\infty}\df{\Realp\{A_{\rho}(x)-Q_{\rho}(x)\}}{x^{\theta}}&=-\infty. \label{omega5}
 \end{align}
 If in assumption \eqref{omega2}, we replace $\Realp\,b(n_k)$ by  $\Imp\,b(n_k)$,  then \eqref{omega4} and \eqref{omega5} remain valid.
 \end{theorem}

The conclusion \eqref{omega4} is equivalent to the following statement: There exists a sequence $\{x_n\}$ tending to $\infty$, such that there does not exist any positive constant $C$ such that, for all $n$,
$$\Realp\{A_{\rho}(x_n)-Q_{\rho}(x_n)\}\leq Cx_n^{\theta}. $$
In such a situation, we write
$$\Realp\{A_{\rho}(x_n)-Q_{\rho}(x_n)\}=\Omega_+(x^{\theta}).$$
Similar remarks hold for \eqref{omega5}.

 \begin{theorem}\label{bigO}\cite[p.~106, Theorem 4.1]{annals2}
 Suppose that the functional equation
\begin{equation*}\label{fe1}
 \Delta(s)\varphi(s)=\Delta(\delta-s)\psi(\delta-s)
 \end{equation*}
 is satisfied with $\delta>0$, and that $\varphi(s)$ is an entire function. Then, if $A(x)$ is defined by \eqref{omega6} and $Q(x)$ is defined by \eqref{poles}, as $x\to\infty$,
 \begin{equation}\label{bigO1}
 A(x)-Q(x)=O\left(x^{\delta/2-1/(4A)+2Au\eta}\right)+O\left(\sum_{x<\lambda_n\leq x^{\prime}}|a(n)|\right),
 \end{equation}
 for every $\eta\geq0$, where
 \begin{equation}\label{bigO3}
 u:=\beta-\frac12\delta-\frac{1}{4A},
 \end{equation}
 and $\beta$ is chosen so that $\sum_{n=1}^{\infty}|b(n)|\mu_n^{-\beta}$ converges. Furthermore,
 \begin{equation}\label{bigO2}
  x^{\prime}=x+O(x^{1-\eta-1/(2A)}).
 \end{equation}
   \end{theorem}

  %In their general $\Omega$  and ``big O'' theorems,  Chandrasekharan and Narasimhan consider the more general Riesz sum $\sum_{\lambda_n\leq x}a(n)(x-\lambda_n)^{\rho}$ and they do not assume that the analytic continuation of $\varphi(s)$ is an entire function. However, in our theorems and conjectures, the analytic continuation of $\varphi(s)$ is an entire function and $\rho=0$.

 \section{The First $\Omega$ and ``Big O'' Theorems and Conjecture}\label{first}
%When appealing to a theorem of Chandrasekharan and Narasimhan below, we use the notation for Theorems \ref{omega1} and \ref{bigO}.  In the notation of  \cite{bessel2bbskaz}, let $\theta=a/p$ and $\sigma=b/q$, where $p$ and $q$ are odd primes.
 Let $\chi_1$ and $\chi_2$ be primitive, non-principal even characters modulo  $p$ and $q$, respectively.
Let $\tau(\chi_1)$ and $\tau(\chi_2)$ denote their corresponding Gauss sums.  Lastly, we use the notation
\begin{equation}\label{chichi}
d_{\chi_1,\chi_2}(n)=\sum_{d|n}\chi_1(d)\chi_2(n/d).
\end{equation}

\begin{theorem}\label{at} Assume that $\chi_1$ and $\chi_2$ are non-principal  even characters modulo the primes $p$ and $q$, respectively. Let
\begin{equation}\label{D}
 \mathbb{D}_{\chi_1,\chi_2}(x):=\sum_{n\leq x}d_{\chi_1,\chi_2}(n).
 \end{equation}
%Suppose there exists a subset $n_k$, $1\leq n_k<\infty$, such that
%\begin{equation}\label{a}
%\sum_{n_k=1}^{\infty}\df{|\Realp\{\tau(\chi_1)\tau(\chi_2)d_{\overline{\chi%}_1,\overline{\chi_2}}(n_k)\}|}{n_k^{3/4}}=+\infty.
%\end{equation}
Then
\begin{align}
\varlimsup_{x\to\infty}\df{\Realp{\, \mathbb{D}_{\chi_1,\chi_2}(x)}}{x^{1/4}}&=+\infty, \label{aa}\\
\varliminf_{x\to\infty}\df{\Realp{\, \mathbb{D}_{\chi_1,\chi_2}(x)}}{x^{1/4}}&=-\infty. \label{ab}
\end{align}
Both \eqref{aa} and \eqref{ab} remain valid if we replace $\Realp$ by $\Imp$ in each of them.
\end{theorem}

\begin{proof} Recall that \cite[p.~74]{bessel2bbskaz} (if $2s$ is replaced by $s$)
\begin{gather}
\left(\df{\pi^2}{pq}\right)^{-s/2}\Gamma^2\left(\df12 s\right)L(s,\chi_1)L(s,\chi_2)\notag\\=
\df{\tau(\chi_1)\tau(\chi_2)}{\sqrt{pq}}\left(\df{\pi^2}{pq}\right)^{-(1-s)/2}
\Gamma^2\left(\df12(1-s)\right)L(1-s,\overline{\chi}_1)L(1-s,\overline{\chi}_2).\label{b}
\end{gather}
We now apply Theorem \ref{omega1}.  The parameters from Definition \ref{def} and Theorem \ref{omega1} are:
\begin{gather}
 N=2,\qquad \delta=1,\qquad A=1, \qquad \theta=\df14,\label{p1}\\
a(n)=d_{\chi_1,\chi_2}(n), \qquad b(n)=\tau(\chi_1)\tau(\chi_2)\df{d_{\overline{\chi}_1,\overline{\chi_2}}(n)}{\sqrt{pq}}, \qquad
\lambda_n=\mu_n=\df{\pi n}{\sqrt{pq}}.\label{p2}
\end{gather}
From the functional equation \eqref{b}, since the analytic continuations of $L(s,\chi_1)$ and $L(s,\chi_2)$ are entire functions,  we see that $L(0,\chi_1)=L(0,\chi_2)=0$.  It follows that $Q(x)=0$.
Theorem \ref{at} now follows immediately provided that we can show that \eqref{omega2} holds, that is,
there exists a subset $n_k$, $1\leq n_k<\infty$, such that
\begin{equation}\label{a}
\sum_{n_k=1}^{\infty}\df{|\Realp\{\tau(\chi_1)\tau(\chi_2)d_{\overline{\chi}_1,\overline{\chi_2}}(n_k)\}|}{n_k^{3/4}}=+\infty.
\end{equation}

Recall that $p$ and $q$ are primes.  First, suppose that $\tau(\chi_1)\tau(\chi_2)$ is not purely imaginary.
 Choose $n$ such that $n\equiv1\pmod{p}$ and  $n\equiv1\pmod{q}$.  Thus, we consider the set $S_{pq}:=\{1+mpq, 1\leq m<\infty\}$.
By Dirichlet's theorem on primes in arithmetic progressions, $S_{pq}$ contains an infinite number of primes $\{P_k\}$.  Note that
$$ d_{\overline{\chi}_1,\overline{\chi}_2}(P_m)=2, \qquad m\geq 1.$$
Thus, we have located an infinite subset $n_k=P_k$, where the series terms in \eqref{a} are positive and where the series diverges.

Suppose now that $\tau(\chi_1)\tau(\chi_2)$ is purely imaginary.  Choose $r$ such that $\chi_1(r)+\chi_2(r)$ is not real. The set
$S_{pq}:=\{r+mpq, 1\leq m<\infty\}$ contains an infinite number of primes $\{Q_k\}$ by Dirichlet's theorem. Hence,
$$d_{\overline{\chi}_1,\overline{\chi}_2}(Q_k)=\chi_1(r)+\chi_2(r).$$
is not real.  Now proceed as in the previous case with $n_k=Q_k$.  This completes the proof.
\end{proof}

We provide the motivation for studying $d_{\chi_1,\chi_2}(n)$.  We see from  \cite[p.~78, Equation (3.11)]{bessel2bbskaz} that
\begin{align}%\label{ccc1}
\mathbb{C}\mathbb{C}(x):&={\sum_{nm\leq x}}^{\prime}\cos\Big(\frac{2\pi
na}{p}\Big)\cos\Big(\frac{2\pi
mb}{q}\Big)\label{CC}\\
&=\frac{1}{\phi(p)\phi(q)}\Big(\sum_{\substack{\chi_1\bmod{p}\\
\chi_1\neq \chi_0, \text{ even}}} \sum_{\substack{\chi_2\bmod{q}\\
\chi_2\neq \chi_0, \text{
even}}}\chi_1(a)\chi_2(b)\tau(\overline{\chi}_1)
\tau(\overline{\chi}_2){\sum_{n\leq x}}^{\prime}d_{\chi_1, \chi_2}(n)\Big) \nonumber \\
&\quad -\frac{1}{\phi(p)}{\sum_{n\leq x}}^{\prime}\left[\frac{x}{n}\right]\cos \Big(\frac{2\pi nb}{q}\Big)
-\frac{1}{\phi(q)}{\sum_{n\leq x}}^{\prime}\left[\frac{x}{n}\right]\cos
\Big(\frac{2\pi
na}{p}\Big) \nonumber\\
&\quad
+\frac{p}{\phi(p)}{\sum_{pn\leq x}}^{\prime}\left[\frac{x}{pn}\right]\cos
\Big(\frac{2\pi nb}{q}\Big)
+\frac{q}{\phi(q)}{\sum_{qn\leq x}}^{\prime}  \left[\frac{x}{qn}\right]\cos \Big(\frac{2\pi na}{p}\Big) \nonumber\\
&\quad-\frac{1}{\phi(p)\phi(q)}\Big({\sum_{ n\leq
x}}^{\prime}d(n) -q{\sum_{ n\leq
x/q}}^{\prime}d(n)-p{\sum_{ n\leq x/p}}^{\prime}d(n)
+pq{\sum_{\leq x/pq}}^{\prime}d(n) \Big), \label{ccc1}
\end{align}
where $\phi(n)$ denotes Euler's $\phi$-function.
%Appearing on the far right-hand side of \eqref{ccc1}, there are  multiple copies of sums of the forms
%\begin{equation*}
%{\sum_{n\leq x}}^{\prime}d_{\chi_1, \chi_2}(n)\qquad \text{and} %\qquad{\sum_{n\leq \mu x}}^{\prime}d(n).
%\end{equation*}
 The sums
 $${\sum_{n\leq x}}^{\prime}d_{\chi_1, \chi_2}(n)$$
 in \eqref{ccc1} were examined in Theorem \ref{at}.
 %When the identity \eqref{b.c.11} is applied to each of the four sums immediately following the character sums in \eqref{ccc}, we find that the terms not involving the doubly infinite series cancel.
 To examine the next four sums
  recall the identity \cite[p.~2063]{besselII}
 \begin{equation}\label{2063}
 {\sum_{n\leq x}}^{\prime}\left[\df{x}{n}\right]
 \cos\left(\frac{2 \pi na}{q}\right)=
 {\sum_{n\leq x/q}}^{\prime}d(n)+\sum_{\substack{d|q \\ d>1}}\df{1}{\phi(d)}
 \sum_{\substack{\chi \,\text{mod}\, d\\\chi\, \text{even}}}
 \chi(a)\tau(\overline{\chi}){\sum_{n\leq dx/q}}^{\prime}d_{\chi}(n).
 \end{equation}
 where $\phi(d)$ denotes Euler's $\phi$-function, and where
 \begin{equation}\label{dchi}
 d_{\chi}(n)=\sum_{d|n}\chi(d).
 \end{equation}
  Chandrasekharan and Narasimhan \cite[p.~133]{annals2} established an  $\Omega$ theorem for
  \begin{equation}\label{133}
 D_{\chi}(x):={\sum_{n\leq x}}^{\prime}d_{\chi}(n)
 \end{equation}
 analogous to Theorem \ref{at}.
  For the last four sums on the right-hand side of \eqref{ccc1}, an $\Omega$  theorem is found in \eqref{d(n)omega}.   In summary, although we can apply an $\Omega$ theorem to each of the divisor sums appearing in \eqref{ccc1}, we cannot apply Theorem \ref{omega1} and the aforementioned analogue from \cite{annals2} to a sum of these arithmetic sums for which we have $\Omega$ theorems.  Thus, we formulate the following conjecture.

\begin{conjecture} If $\mathbb{C}\mathbb{C}(x)$ is defined  in \eqref{CC}, then
\begin{align*}
\varlimsup_{x\to\infty}\df{ \mathbb{C}\mathbb{C}\,(x) }{x^{1/4}}&=+\infty, \\
\varliminf_{x\to\infty}\df{ \mathbb{C}\mathbb{C}\,(x)}{x^{1/4}}&=-\infty. 
\end{align*}
\end{conjecture}
%Alternatively,
%\begin{equation}\label{omega1a}
%\mathbb{C}\mathbb{C}(x)=\Omega_{\pm}(x^{1/4}),
%\end{equation}
%as $x\to\infty$.

\begin{theorem}\label{O1} For each $\epsilon>0$, as $x\to \infty$,
$$\mathbb{D}_{\chi_1,\chi_2}(x)=O(x^{1/3+\epsilon}).$$
\end{theorem}

\begin{proof} To prove Theorem \ref{O1}, we apply Theorem \ref{bigO}.   For the parameters that are needed, we refer to \eqref{p1} and \eqref{p2}.  First, recall that $\delta = A =1$. We need the series $\sum_{n=1}^{\infty}|b(n)|n^{-\beta}$ to converge, and so we take $\beta=1+\epsilon$, for every $\epsilon>0$. (The parameter $\epsilon$ will not necessarily be the same with each occurrence.)
Thus, from \eqref{bigO3},
$$u:=\beta-\frac12\delta-\frac{1}{4A}=\frac14+\epsilon.$$
The first ``Big O'' power in \eqref{bigO1} is then
$$ \frac14+\left(\frac12+2\epsilon\right)\eta,$$
where $\eta\geq0$ and is yet to be determined.
We also need to determine the order of
$$\sum_{x<\lambda_n\leq x^{\prime}}|a(n)|, $$
where $a(n)$ is given in \eqref{p2} and $x^{\prime}=x+O(x^{1/2-\eta})$ is defined in \eqref{bigO2}.
  Trivially, $|a(n)|\leq d(n)$. By \eqref{d(n)bigO}, an upper bound for this sum is $O(x^{1/2-\eta}\log x)$.  In order to achieve the most effective upper bound, we find the solution to
\begin{equation}\label{etaeta}\frac14+ \left(\frac12+2\epsilon\right)\eta=\frac12-\eta+\epsilon^{\prime},
\end{equation}
where $\epsilon, \epsilon^{\prime}$ are arbitrarily small positive numbers.
Thus, our optimal choice is $\eta=\tf16$.  Hence, by \eqref{etaeta}, our proof of Theorem \ref{O1} is complete.
\end{proof}

\begin{theorem}\label{222} For each $\epsilon>0$, as $x\to \infty$,
\begin{equation*}\label{1/3-}
\mathbb{C}\mathbb{C}(x)=O(x^{1/3+\epsilon}).
\end{equation*}
\end{theorem}

\begin{proof}
Again, appearing in \eqref{ccc1}, there are three kinds of divisor sums.  For the former sums, we have the bound in Theorem \ref{O1}.
 For the ``middle'' four sums, we recall \eqref{2063}. Since $|d_{\chi}(n)|
 \leq d(n)$, we can apply the bound given in \eqref{d(n)bigO} for each of these four ``middle'' sums.
  For the latter four sums, we can invoke \eqref{d(n)bigO}.
 In summary, each of the sums in \eqref{ccc1} has the bound expressed in Theorem \ref{222}, and so the proof is complete.
\end{proof}

\section{The Second $\Omega$ and ``Big O'' Theorems and Conjecture}\label{second}
%Except when appealing to Chandrasekharan and Narasimhan's Theorem \ref{bigO},  we use the notation for Theorems 4.1 and 4.2 in \cite{bessel2bbskaz}.
 %In particular, let $\theta=a/p$ and $\sigma=b/q$. In Theorem \ref{bt} below we assume these values of $\theta$ and $\sigma$.
 Let $\chi_1$ and $\chi_2$ be non-principal, primitive even and odd characters modulo  $p$ and $q$, respectively.  Let $\tau(\chi_1)$ and $\tau(\chi_2)$ denote the corresponding Gauss sums.    Lastly, recall that $d_{\chi_1,\chi_2}(n)$ is given by \eqref{chichi}.

\begin{theorem}\label{bt} Assume that $\chi_1$ is a non-principal primitive even character modulo $p$ and that $\chi_2$ is a non-principal primitive odd character modulo $q$, where $p$ and $q$ are primes. Let
$\mathbb{D}_{\chi_1,\chi_2}(x)$ be defined by \eqref{D}.
%Suppose that  there exists a subset $n_k$, $1\leq n_k<\infty$, such that
%\begin{equation}\label{bbb}
%\sum_{n_k=1}^{\infty}\df{|\Realp\{\tau(\chi_1)\tau(\chi_2)d_{\overline{\chi%}_1,\overline{\chi}_2}(n_k)\}|}{n_k^{3/4}}=\infty,
%\end{equation}
Then
\begin{align}
\varlimsup_{x\to\infty}\df{\Realp{\, \mathbb{D}_{\chi_1,\chi_2}(x)}}{x^{1/4}}&=+\infty, \label{aa1}\\
\varliminf_{x\to\infty}\df{\Realp{\, \mathbb{D}_{\chi_1,\chi_2}(x)}}{x^{1/4}}&=-\infty. \label{ab1}
\end{align}
Both \eqref{aa1} and \eqref{ab1} remain valid if we replace $\Realp$ by $\Imp$, respectively, in \eqref{aa1} and \eqref{ab1}.
\end{theorem}

\begin{proof} Recall from \cite[p.~82]{bessel2bbskaz} that
\begin{gather*}
\Big(\frac{2\pi}{\sqrt{pq}}\Big)^{-s}\Gamma(s)L(s,
\chi_1)L(s,\chi_2)\\
=-\frac{i\tau(\chi_1)\tau(\chi_2)}{\sqrt{pq}}\Big(\frac{2\pi}{\sqrt{pq}}\Big)^{s-1}
\Gamma(1-s)L(1-s,\overline{\chi_1})L(1-s,\overline{\chi}_2).
\end{gather*}
We now apply Theorem \ref{bigO}.  The parameters from Definition \ref{def} and \eqref{omega3} are:
\begin{gather*}
N=1,\qquad \delta=1,\qquad A=1, \qquad \theta=\df14,\\
a(n)=d_{\chi_1,\chi_2}(n), \qquad b(n)=-i\tau(\chi_1)\tau(\chi_2)\df{d_{\overline{\chi}_1,\overline{\chi}_2}(n)}{\sqrt{pq}}, \qquad
\lambda_n=\mu_n=\df{2\pi n}{\sqrt{pq}}.
\end{gather*}
Since $L(0,\chi_1)=0$, and both $L(s,\chi_1)$ and $L(s,\chi_2)$ are entire functions, then $Q(x)=0$.
Theorem \ref{bt} now follows immediately provided that we can show that \eqref{omega2} holds.  The proof that \eqref{omega2} is valid is exactly the same as in the previous theorem.
\end{proof}

Our motivation for studying $d_{\chi_1,\chi_2}(n)$ is similar to that for Theorem \ref{at}. From \cite[p.~85, Equation (4.5)]{bessel2bbskaz}
\begin{align}\label{cs1}
\mathbb{C}\mathbb{S}(x):=&{\sum_{nm\leq x}}^{\prime}\cos\Big(\frac{2\pi
na}{p}\Big)\sin\Big(\frac{2\pi
mb}{q}\Big)\notag\\
=&\frac{1}{i\phi(p)\phi(q)}\sum_{\substack{\chi_1\bmod{p}\\
\chi_1\neq \chi_0, \text{ even}}}
\sum_{\substack{\chi_2\bmod{q}\\
\chi_2\text{ odd}}}\chi_1(a)\chi_2(b)\tau(\overline{\chi_1})
\tau(\overline{\chi_2}){\sum_{n\leq x}}^{\prime}d_{\chi_1, \chi_2}(n) \nonumber \\
&
-\frac{1}{\phi(p)}{\sum_{m\le x}}^{\prime}\left[\frac{x}{m}\right]\sin
\Big(\frac{2\pi mb}{q}\Big)
+\frac{p}{\phi(p)}{\sum_{m\leq x}}^{\prime}\left[\frac{x}{pm}\right]\sin
\Big(\frac{2\pi mb}{q}\Big).
\end{align}
As in  our study of $\mathbb{C}\mathbb{C}(x)$, we observe that multiple sums of the form
$${\sum_{n\leq x}}^{\prime} d_{\chi_1,\chi_2}(n)$$
arise. For and $\Omega$ theorem for  the first set of sums on the right-hand side of \eqref{cs1} appeal to  Theorem \ref{bt}.  For the second set of sums, we recall the identity \cite[p.~2068, Lemma 11]{besselII}
\begin{equation}\label{ssss}
    {\sum_{n\leq x}}^{\prime}\left[\df{x}{n}\right]
    \sin\left(\df{2\pi na}{p}\right)
    =-i\sum_{\substack{d|q\\d>1}}\df{1}{\phi(d)}
    \sum_{\substack{\chi \text{ mod } d\\\chi \text{ odd }}}\chi(a)\tau(\overline{\chi})
    {\sum_{n\leq dx/q}}^{\prime}d_{\chi}(n).
    \end{equation}
    Recall that the character sums on the far right-hand side of \eqref{ssss} were defined in \eqref{dchi} and \eqref{133}.  An $\Omega$ theorem analogous to Theorem \ref{bt} can also be established \cite[p.~133]{annals2}.
We cannot appeal directly to Theorem \ref{bt} and the aforementioned analogue in order to establish an $\Omega$ theorem for $\mathbb{C}\mathbb{S}(x)$.  We therefore must content ourselves to making the following conjecture.

\begin{conjecture} If $\mathbb{C}\mathbb{S}(x)$ is defined on the far left-hand side of \eqref{cs1}, then
\begin{align*}
\varlimsup_{x\to\infty}\df{ \mathbb{C}\mathbb{S}\,(x) }{x^{1/4}}&=+\infty, \\
\varliminf_{x\to\infty}\df{ \mathbb{C}\mathbb{S}\,(x)}{x^{1/4}}&=-\infty. \
\end{align*}
\end{conjecture}

\begin{theorem}\label{cscs} For each $\epsilon>0$, as $x\to \infty$,
$$\mathbb{D}_{\chi_1,\chi_2}(x)=O(x^{1/3+\epsilon}).$$
\end{theorem}

\begin{proof} Because the values of $A$ and $\delta$ are identical to those in the proof of Theorem \ref{O1}, the proof of Theorem \ref{cscs} is the same as that for Theorem \ref{O1}.
\end{proof}

\begin{theorem}\label{O3} For each $\epsilon>0$, as $x\to \infty$,
$$\mathbb{C}\mathbb{S}(x)=O(x^{1/3+\epsilon}).$$
\end{theorem}

\begin{proof}  Referring to \eqref{cs1}, we see that each member of the first set of sums on the right-hand side of \eqref{cs1} satisfies the bounds of Theorem \ref{cscs}.  For the second set of sums, refer to \eqref{ssss}.
    Since $|d_{\chi}(n)|\leq d(n)$, each of the divisor sums on the right-hand side of \eqref{cs1} has an upper bound given by \eqref{d(n)bigO}.  Hence, the proof of Theorem \ref{O3} is complete.
\end{proof}

\section{The Third $\Omega$ and ``Big O" Theorems and Conjecture}\label{third}

The third general $\Omega$ theorem is similar to Theorems \ref{at} and \ref{bt}.

\begin{theorem}\label{ct} Assume that $\chi_1$ and $\chi_2$ are non-principal primitive odd characters modulo $p$ and $q$, respectively. Define
 \begin{equation}\label{D^*}
 {\mathbb{D}}^*_{\chi_1,\chi_2}(x):={\sum_{n\leq x}}^{\prime}nd_{\chi_1,\chi_2}(n).
 \end{equation}
%If  there exists a subset $n_k$, $1\leq n_k<\infty$, such that
%\begin{equation}\label{bbbb}
%\sum_{n_k=1}^{\infty}\df{|\Realp\{\tau(\chi_1)\tau(\chi_2)n_k %d_{\overline{\chi}_1,\overline{\chi_2}}(n_k)\}|}{n_k^{3/4}}=\infty,
%\end{equation}
Then
\begin{align}
\varlimsup_{x\to\infty}\df{\Realp{\, {\mathbb{D}}^*_{\chi_1,\chi_2}(x)}}{x^{5/4}}&=+\infty, \label{aa11}\\
\varliminf_{x\to\infty}\df{\Realp{\, {\mathbb{D}}^*_{\chi_1,\chi_2}}^*(x)}{x^{5/4}}&=-\infty. \label{ab11}
\end{align}
Both \eqref{aa11} and \eqref{ab11} remain valid if we replace $\Realp$ by $\Imp$ in each of \eqref{aa11} and \eqref{ab11}.
\end{theorem}

\begin{proof}
The relevant functional equation is \cite[top of page 89]{bessel2bbskaz} (with $s$ replaced by $\tf12 s$),
\begin{align*}
&\frac{\pi^{-s}}{(pq)^{-s/2}}\Gamma^2\left(\df12 s\right)L(s-1, \chi_1)L(s-1,\chi_2)\\
&\qquad \qquad
=-\frac{\tau(\chi_1)\tau(\chi_2)}{\sqrt{pq}}\frac{\pi^{-(3-s)}}{(pq)^{-(3-s)/2}}
\Gamma^2\left(\dfrac{1}{2}(3-s)\right)L(2-s,\overline{\chi}_1)L(2-s,\overline{\chi}_2).
\end{align*}
Note that
$$L(s-1,\chi_1)L(s-1,\chi_2)=\sum_{n=1}^{\infty}\frac{\chi_1(n)}{n^{s-1}}
\sum_{m=1}^{\infty}\frac{\chi_2(m)}{m^{s-1}}
=\sum_{n=1}^{\infty}\frac{nd_{\chi_1, \chi_2}(n)}{n^{s}}.$$
In the notation of Definition \ref{def} and \eqref{omega3},
\begin{gather}
N=2,\qquad \delta=3,\qquad A=1, \qquad \theta=\df54,\label{3a}\\
a(n)=nd_{\chi_1,\chi_2}(n), \qquad b(n)=-\tau(\chi_1)\tau(\chi_2)\df{nd_{\overline{\chi}_1,\overline{\chi}_2 }(n)}{\sqrt{pq}}, \qquad
\lambda_n=\mu_n=\df{\pi n}{\sqrt{pq}}.\notag\label{3b}
\end{gather}
(In \cite[p.~89]{bessel2bbskaz} the factor $n$ is missing from the definition of $a(n)$.)  Since $\chi_1$ and $\chi_2$ are odd primitive characters, $L(2s-1,\chi_1)$ and $L(2s-1,\chi_2)$ are both entire functions of $s$ which vanish at $s=0$.  Hence, $Q(x)=0$.
Theorem \ref{ct} now follows if we can show that \eqref{omega2} is valid.  This can be shown in the same way as given in the proof of Theorem \ref{at}.
\end{proof}

Let
\begin{equation}\label{ss}
\mathbb{S}\mathbb{S}(x):={\sum_{mn\leq x}}^{\prime}mn\sin(2\pi na/p)\sin(2\pi mb/q).
\end{equation}
%Unfortunately, as above, we can only offer a conjectured $\Omega$-theorem %for $\mathbb{S}\mathbb{S}(\sigma,\theta,x).$
From \cite[p.~91, Equation (5.13)]{bessel2bbskaz},
\begin{equation}\label{v1}
\mathbb{S}\mathbb{S}(x)
=-\frac{1}{\phi(p)\phi(q)}\sum_{\substack{\chi_1\bmod{p}\\ \chi_1
\text{ odd}}} \sum_{\substack{\chi_2\bmod{q}\\ \chi_2 \text{
odd}}}\chi_1(a)\chi_2(b)\tau(\overline{\chi}_1)
\tau(\overline{\chi}_2){\mathbb{D}}^*_{\chi_1,\chi_2}(x),
\end{equation}
where ${\mathbb{D}}^*_{\chi_1,\chi_2}(x)$ is defined in \eqref{D^*}.
Thus, since $\mathbb{S}\mathbb{S}(x)$ is a linear combination of the sums ${\mathbb{D}}^*_{\chi_1,\chi_2}(x)$,  we cannot appeal directly to Theorem \ref{omega1}.  Thus, we make the following conjecture.

\begin{conjecture}\label{3tttt} If $\mathbb{S}\mathbb{S}(x)$ is defined by \eqref{ss}, then
\begin{align*}
\varlimsup_{x\to\infty}\df{ \mathbb{S}\mathbb{S}(x) }{x^{5/4}}&=+\infty,\\
\varliminf_{x\to\infty}\df{ \mathbb{S}\mathbb{S}(x)}{x^{5/4}}&=-\infty. \label{cc222}
\end{align*}
\end{conjecture}

We now use Chandrasekharan and Narasimhan's  Theorem \ref{bigO} to obtain an upper bound for ${\mathbb{D}}^*_{\chi_1,\chi_2}(x)$. From \eqref{3a}, $A=1$ and $\delta=3$. Observe that $|a(n)|\leq nd(n)$. Furthermore, for some constant $C>0$, we also see that $|b(n)|\leq Cnd(n)$. Thus, for each $\epsilon >0$, we shall take $\beta=2+\epsilon$.  Hence, from \eqref{bigO3},
\begin{equation*}\label{3c}
u=2+\epsilon-\frac32-\frac14=\frac14+\epsilon.
\end{equation*}
With a reference to \eqref{bigO1}, we need to calculate, for $\eta\geq0$,
\begin{equation}\label{3d}
\df{1}{2}\delta-\df{1}{4A}+2Au\eta =\dfrac{3}{2}-\df14+2\left(\df14+\epsilon\right)\eta=\df54+\left(\df12+2\epsilon\right)\eta.
\end{equation}
For the second power in \eqref{bigO1}, we use partial summation and \eqref{d(n)bigO} to deduce that
\begin{equation}\label{3e}
\sum_{x< \mu_n\leq x+O(x^{1/2-\eta})}nd(n) =O(x^{3/2-\eta}\log x).
\end{equation}
From \eqref{3d} and \eqref{3e}, we seek the optimal power of $x$ by solving
\begin{equation*}
\df54+\left(\df12+2\epsilon\right)\eta=\df32-\eta+\epsilon.
\end{equation*}
Solving this simple equation, we see that $\eta=\tf16+\epsilon$. Therefore, we have established the following theorem.

\begin{theorem}\label{3ttt}
As $x\to \infty$, for every $\epsilon>0$,
\begin{equation*}\label{3f}
{\mathbb{D}}^*_{\chi_1,\chi_2}(x)=O(x^{4/3+\epsilon}),
\end{equation*}
where ${\mathbb{D}}^*_{\chi_1,\chi_2}(x)$ is defined by \eqref{D^*}.
\end{theorem}

Using \eqref{v1} and Theorem \ref{3ttt}, we can immediately  deduce the following theorem.

\begin{theorem}\label{3tt}
As $x\to \infty$, for every $\epsilon>0$,
\begin{equation*}
\mathbb{S}\mathbb{S}(x)=O(x^{4/3+\epsilon}).
\end{equation*}
\end{theorem}

We conclude this section with a special case. Let $\tf{a}{p}=\tf{b}{q}=\tf14$. Then
$$\sin(2\pi n/4)=\begin{cases}(-1)^{(n-1)/2}, \quad &n \text{ odd},\\
0, &n \text{ even},
\end{cases}
$$
and
\begin{align*}\label{6}
 \mathbb{S}\mathbb{S}(\tf14,\tf14,x):=&-\sum_{\substack{mn\leq x\\m,n \text{ odd}}}mn(-1)^{(m+n)/2}\notag\\
 =&\sum_{(2j+1)(2k+1)\leq x}(-1)^{j+k}(2j+1)(2k+1),
 \end{align*}
 where we set $m=2j+1$, $n=2k+1$. This is a rather interesting lattice point problem.  We are counting lattice points under the hyperbola $ab\leq x$, but we require both coordinates to be odd, and we put a weight on them.

 We restate Conjecture \ref{3tttt} and Theorem \ref{3tt} in this particular case $\tf{a}{p}=\tf{b}{q}=\tf14$.

\begin{conjecture}If $\mathbb{S}\mathbb{S}(x)$ is defined by \eqref{ss}, then
\begin{align*}
\varlimsup_{x\to\infty}\df{ \mathbb{S}\mathbb{S}(\tf14, \tf14,x) }{x^{5/4}}&=+\infty, \\%\label{ccc111}\\
\varliminf_{x\to\infty}\df{ \mathbb{S}\mathbb{S}(\tf14,\tf14,x)}{x^{5/4}}&=-\infty. %\label{ccc222}
\end{align*}
\end{conjecture}

\begin{theorem}
As $x\to \infty$, for every $\epsilon>0$,
\begin{equation*}
\mathbb{S}\mathbb{S}(\tf14,\tf14,x)=O(x^{4/3+\epsilon}).
\end{equation*}
\end{theorem}

%I searched Dickson's book \cite[Chapter 10]{dickson}, but I could not find sums like those in \eqref{6}.

\section{Sums with a Product of an Arbitrary Number of $\sin$'s}\label{fourth} 
We begin this section with a definition. Let $J_{\nu}(x)$ denote the ordinary Bessel function of order $\nu$.
Define
\begin{align}\label{main}
K_{\nu}(x;\mu;m):=&\int_0^{\infty}u_{m-1}^{\nu-\mu-1}J_{\mu}(u_{m-1})du_{m-1}
\int_0^{\infty}u_{m-2}^{\nu-\mu-1}J_{\mu}(u_{m-2})du_{m-2}\notag \\
&\cdots\cdots\int_0^{\infty}u_{1}^{\nu-\mu-1}J_{\mu}(u_1)J_{\nu}(x/u_1u_2\cdots u_{m-1})du_{1},
\end{align}
provided that $\mu,\nu>-3/2$, so that the integral converges.

In the sequel, we apply Theorems 2 and 4 of
\cite[pp.~351, 356]{identitiesI}, which, for convenience, we offer below.

\begin{theorem}\label{identities351} Let $\varphi(s)$ and $\psi(s)$ satisfy the functional equation \eqref{Delta} with $\Delta(s)=\Gamma^k(s)$. If $k\geq2$, suppose that $\delta>-\tf12$. Assume that $\rho>2k\sigma_a^{*}-k\delta-\tf12$, where $\sigma_a^{*}$ is the abscissa of absolute convergence of $\psi(s)$.  If $x>0$, then
\begin{gather*}\label{identities351a}
\df{1}{\Gamma(\rho+1)}{\sum_{\lambda_n\leq x}}^{\prime}a(n)(x-\lambda_n)^{\rho}\notag\\
=2^{\rho(1-k)}\sum_{n=1}^{\infty}b(n)\left(\df{x}{\mu_n}\right)^{(\delta+\rho)/2}K_{\delta+\rho}(2^k(\mu_n x)^{1/2};\delta-1;k)+Q_{\rho}(x),
\end{gather*}
where $Q_{\rho}(x)$ is defined in \eqref{poles}.
\end{theorem}

The following Theorem \ref{identities356} is an extension of Theorem \ref{identities351}.  In our application, the hypotheses are readily verified.

\begin{theorem}\label{identities356} Suppose that for $\sigma>{\sigma_a}^*$,
\begin{equation*}
\sup_{0\leq h\leq 1}\left|\sum_{r^{2k}\leq \mu_n\leq(r+h)^{2k}}b(n)\mu_n^{\sigma -1/(2k)}\right|=o(1),
\end{equation*}
as $r\to\infty$. Then \eqref{identities351a} is valid for $q>2k\sigma_a^*-k\delta-\tf32$ and for those positive values of $x$ such that the left-hand side of \eqref{identities351a} is defined. The series on the right-hand side of \eqref{identities351a} converges uniformly on any interval for $x>0$ where the left-hand side is continuous.  The convergence is bounded on any interval $0<x_1\leq x\leq x_2$ when $\rho=0$.
\end{theorem}

Let $k$ be an arbitrary positive integer.  Let $\chi_1,\chi_2, \dots, \chi_k$ be odd, primitive, non-principal, characters modulo $p_1,p_2, \dots, p_k$, respectively.  Throughout this section, we assume that  $\chi_1,\chi_2, \dots, \chi_k$ and $\chi$ are odd, and so we normally  make this assumption without comment.  Let $a_1, a_2, \dots , a_k$ denote    positive integers such that $(a_j,p_j)=1$, $1\leq j\leq k$.

\begin{definition}\label{defchichi}  Let $n_1,n_2,\dots , n_k$ denote positive integers. Define
\begin{equation*}
d_{\chi_1,\chi_2,\dots,\chi_k}(n):=
\sum_{n_1n_2\cdots n_k=n}\chi_1(n_1)\chi_2(n_2) \cdots \chi_k(n_k)\label{sum1chichi}
\end{equation*}
and
\begin{equation}\label{sumchichi}
\mathbb{D}^k(x):={\sum_{n\leq x}}^{\prime} nd_{\chi_1,\chi_2,\dots,\chi_k}(n),
\end{equation}
where, as customary, the prime $\prime$ on the summation sign indicates that if $x$ is an integer, only $\tf12$ of the term is counted.
\end{definition}

Note that if $k=2$, then \eqref{sumchichi} is identical to \eqref{D^*}.

%Appearing in our primary theorem below is the following multiple integral with
 % Let $\rho\geq0$, and let $r_1$ denote a positive integer.
%Define
%\begin{align}\label{main}
%K_{\rho+1/2}(x;\ell;r_1+1):=&\int_0^{\infty}u_{r_1}^{\rho}J_{\ell}(u_{r_1})du_{r_1}
%\int_0^{\infty}u_{r_1-1}^{\rho}J_{\ell}(u_{r_1-1})du_{r_1-1}\notag \\
%&\cdots\cdots\int_0^{\infty}u_{1}^{\rho}J_{\ell}(u_1)J_{\rho+1/2}(x/u_1u_2\cdots u_{r_1})du_{1}.
%\end{align}
%A special case of \eqref{main} appears in the following theorem.

\begin{theorem}\label{theorem1}
Recall the definition of $K_{\rho+3/2}(x;\tf12;k)$ defined by \eqref{main}, as well as the definitions and notation above.  Then for $\rho>(k-3)/2$,
\begin{gather*}
\sideset{}{'}\sum_{n\le x} n d_{\chi_1,\chi_2,\dots,\chi_k}(n)(x^2-n^2)^{\rho}=\df{(-i)^k(p_1p_2\cdots p_k)^{\rho-1/2}\tau(\chi_1)\tau(\chi_2)\cdots\tau(\chi_k)}{\pi^{k\rho}}\notag\\
\times\sum_{n=1}^{\infty}n d_{\overline{\chi}_1,\overline{\chi}_2,\dots,\overline{\chi}_k}(n)
\left(\df{x}{n}\right)^{\rho+3/2}K_{\rho+3/2}\left(\df{2^k\pi^knx}{p_1p_2\cdots p_k};\df12;k\right).\label{equa1}
\end{gather*}
%where the prime $\prime$ on the summation sign at the left indicates that %if $x$ is a positive integer, then the last term is multiplied by $\tf12$.
\end{theorem}

The integral
$$K_{\rho+3/2}\left(\df{2^k\pi^knx}{p_1p_2\cdots p_k};\tf12;k\right)$$
can be represented by a Meijer $G$-function, but we do not provide it here. 
%is evaluated  in Theorem \ref{meijerevaluation}  in terms of the Meijer $G$-function.

\begin{proof}

Recall that the Dirichlet $L$-function $L(s,\chi)$ of modulus $q$ satisfies the functional equation \cite[p.~82]{bessel2bbskaz}
\begin{equation}\label{funcequa}
\left(\df{\pi}{q}\right)^{-(2s+1)/2}\Gamma\left(s+\df12\right)L(2s,\chi)=
-\df{i\tau(\chi)}{\sqrt{q}}\left(\df{\pi}{q}\right)^{-(1-s)}\Gamma\left(1-s\right)L(1-2s,\overline{\chi}).
\end{equation}
We replace $s$ by $s-\tf12$, $q$ by $p_j$, and   $\chi$ by $\chi_j$ in \eqref{funcequa}, $1\leq j\leq k$. Multiply the $k$ functional equations together to obtain
\begin{gather}
\df{\pi^{-ks}}{(p_1p_2\cdots p_k)^{-s}}\Gamma^k\left(s\right)L(2s-1,\chi_1)L(2s-1,\chi_2)\cdots L(2s-1,\chi_k)\notag\\
=\df{(-i)^k\tau(\chi_1)\tau(\chi_2)\cdots\tau(\chi_k)}{\sqrt{p_1p_2\cdots p_k}}\df{\pi^{-k(3-2s)/2}}{(p_1p_2\cdots p_k)^{-(3-2s)/2}}
\Gamma^k\left(\df{3-2s}{2}\right)\notag\\\times L(2-2s,\overline{\chi}_1)L(2-2s,\overline{\chi}_2)\cdots L(2-2s,\overline{\chi}_k).\label{kfuncequa2}
\end{gather}
Note that $\delta=\tf32$ and that
\begin{gather*}L(2s-1,\chi_1)L(2s-1,\chi_2)\cdots L(2s-1,\chi_k)\\= \sum_{n_1=1}^{\infty}\df{\chi_1(n_1)}{n_1^{2s-1}}
\sum_{n_2=1}^{\infty}\df{\chi_2(n_2)}{n_2^{2s-1}}\cdots\sum_{n_k=1}^{\infty}\df{\chi_k(n_2)}{n_k^{2s-1}}=\sum_{n=1}^{\infty}
\df{n d_{\overline{\chi}_1,\overline{\chi}_2,\dots, \overline{\chi}_k}(n)}{n^{2s}}, \qquad \sigma=\Realp\,s>1.
\end{gather*}

We apply Theorems \ref{identities351} and \ref{identities356}.
Observe that $Q_{\rho}(x)=0$, because, for $1\leq j\leq k$,  $L(s,\chi_j)$ is an entire function and  $L(-1,\chi_j)=0$, since $\chi_j$ is odd.
%In \eqref{kfuncequa}, we see that $\delta=3$, $A=\tf12 k$,
Also observe that
$$ a(n)=n d_{\chi_1,\chi_2,\dots,\chi_k}(n), \qquad b(n)=
\df{(-i)^k\tau(\chi_1)\tau(\chi_2)\cdots\tau(\chi_k)}{\sqrt{p_1p_2\cdots p_k}}
\,n d_{\overline{\chi}_1,\overline{\chi}_2,\dots,\overline{\chi}_k}(n),$$
and
$$ \lambda_n=\mu_n=\df{\pi^k n^2}{p_1p_2\cdots p_k}.$$
In Theorem \ref{identities351} the sum is over $\lambda_n\leq x$.  Replace $x$ by $$\df{\pi^k x^2}{p_1p_2\cdots p_k}.$$  Thus, the amended sum will be over $n\leq x$ and
$$ (x-\lambda_n)^{\rho} \quad \Rightarrow \quad \df{\pi^{k\rho}}{(p_1p_2\cdots p_k)^{\rho}}(x^2-n^2)^{\rho}.$$
  Also appearing in Theorem \ref{identities351} is a quotient in the summands on the right side that will be transformed by the change in variable above, i.e.,
  $$ \left(\df{x}{\mu_n}\right)^{(3/2+\rho)/2} \quad \Rightarrow \quad \left(\df{x}{n}\right)^{3/2+\rho}.$$
Lastly, on the right-hand side of \eqref{identities351a},
$$2^k(\mu_nx)^{1/2}\quad \Rightarrow \quad 2^k\left(\df{\pi^kn^2}{p_1p_2\cdots p_k}\cdot \df{\pi^kx^2}{p_1p_2\cdots p_k}\right)^{1/2}=
\df{2^k\pi^knx}{p_1p_2\cdots p_k}.$$
With all of these substitutions, we deduce that
\begin{gather*}
\df{\pi^{k\rho}}{(p_1p_2\cdots p_k)^{\rho}}\sideset{}{'}\sum_{ n\le x} n d_{\chi_1,\chi_2,\dots,\chi_k}(n)(x^2-n^2)^{\rho}\notag\\=\df{(-i)^k\tau(\chi_1)\tau(\chi_2)\cdots\tau(\chi_k)}{\sqrt{p_1p_2\cdots p_k}}
\sum_{n=1}^{\infty}n d_{\overline{\chi}_1,\overline{\chi}_2,\dots,\overline{\chi}_k}(n)
\left(\df{x}{n}\right)^{\rho+3/2}K_{\rho+3/2}\left(\df{2^k\pi^knx}{p_1p_2\cdots p_k};\df12;k\right),
\end{gather*}
where $K_{\rho+3/2}$ is defined in \eqref{main}. The identity above is precisely \eqref{equa1}, and so the proof is complete.
\end{proof}

We next prove an identity for the sum
\begin{gather}
\mathbb{S}_{\rho}(a_1,a_2,\dots a_k;p_1,p_2,\dots p_k;x):=\notag\\
\sideset{}{'}\sum_{1\le n_1n_2\cdots n_k\le x} n_1n_2\cdots n_k\,\sin(2\pi n_1a_1/p_1)\sin(2\pi n_2a_2/p_2)\cdots \sin(2\pi n_ka_k/p_k)(x^2-n^2)^{\rho},\label{S}
\end{gather}
where $(a_j,p_j)=1, 1\leq j \leq k$ and $n=n_1n_2\cdots n_k.$

\begin{lemma}\label{25}\cite[p.~72]{bessel2bbskaz} Let $(a,q)=1$ and $ n\in \mathbb{Z}$. Suppose that $\chi$ is an odd, primitive, non-principal character of order $q$.  Then
$$\sin(2\pi na/q)=\df{1}{i\phi(q)}\sum_{\substack{\chi\, \mathrm{ mod }\, q\\ \chi \,\mathrm{ odd }\,}}\chi(a)\tau(\overline{\chi})\chi(n),$$
where $\phi(q)$ denotes Euler's $\phi$-function.
\end{lemma}

\begin{lemma}\label{26} \cite[p.~3806, Equation (4.12)]{pams} For any primitive character $\chi$ modulo $q$,
$$
\sum_{\substack{\chi\, \mathrm{ mod }\, q\\ \chi \,\mathrm{ odd }\,}}\chi(a)\overline{\chi}(b)=\begin{cases}\pm\tf12 \phi(q), \quad &\text{if } a\equiv \pm b\pmod q \text{ and } (a,q)=1,\\0, &\text{otherwise}.
\end{cases}$$
\end{lemma}

We let
\begin{equation*}\label{notation} \sum_{\pm n_1,\pm n_2,\dots, \pm n_k}
\end{equation*}
 denote a sum over all $2^k$ pairs
$$n_1\equiv \pm a_1 \pmod{p_1}, n_2\equiv \pm{a_2}   \pmod{p_2}, n_k\equiv \pm a_k \pmod{p_k}.$$

\begin{theorem}\label{theorem2}
%Let $n_1,n_2,\dots, n_k$ denote positive integers. Then
%\begin{equation}\label{notation}\sum_{\pm n_1,\pm n_2,\dots, \pm n_k}
%\end{equation}
 %indicates a sum over all $2^k$ pairs
 In the notation above, for $\rho>(k-3)/2$,
\begin{gather*}
\mathbb{S}_{\rho}(a_1,a_2,\dots, a_k;p_1,p_2,\dots, p_k;x)\\
=\df{x^{\rho+3/2}(p_1p_2\cdots p_k)^{\rho+1/2}}{2^k\pi^{k\rho}}\sum_{\pm n_1,\pm n_2,\dots, \pm n_k}
\df{(-1)^{\textup{sgn}}K_{\rho+3/2}\left(\df{2^k\pi^kn_1n_2\cdots n_k x}{p_1p_2\cdots p_k};\df12;k\right)}{\sqrt{n_1n_2\cdots n_k}},
\end{gather*}
where \textup{sgn} denotes the number of minus signs in a particular $k$-tuple, $\pm n_1,\pm n_2,\dots, \pm n_k.$
\end{theorem}

\begin{proof}
By \eqref{S} and Lemma \ref{25}, if $n=n_1 n_2\cdots n_k$,
\begin{align}\label{a1}
&\mathbb{S}_{\rho}(a_1,a_2,\dots, a_k;p_1,p_2,\cdots,p_k;x)=\df{(-i)^k}{\phi(p_1)\phi(p_2)\cdots\phi(p_k)}
\sideset{}{'}\sum_{1\le n_1n_2\cdots n_k\le x} n_1n_2\cdots n_k(x^2-n^2)^{\rho}\notag\\
&\times\sum_{\chi_1 \text{ mod } p_1}\sum_{\chi_2 \text{ mod } p_2}\cdots\sum_{\chi_k \text{ mod } p_k}
\chi(a_1)\chi(a_2)\cdots\chi(a_k)\tau(\overline{\chi}_1)\tau(\overline{\chi}_2)\cdots\tau(\overline{\chi}_k)\chi(n_1)\chi(n_2)\cdots\chi(n_k)\notag\\
=&\df{(-i)^k}{\phi(p_1)\phi(p_2)\cdots\phi(p_k)}\times\sum_{\chi_1 \text{ mod } p_1}\sum_{\chi_2 \text{ mod } p_2}\cdots\sum_{\chi_k \text{ mod } p_k}
\chi(a_1)\chi(a_2)\cdots\chi(a_k)\\
&\times\tau(\overline{\chi}_1)\tau(\overline{\chi}_2)\cdots\tau(\overline{\chi}_k)\sideset{}{'}\sum_{n\leq x}n d_{\chi_1,\chi_2,\dots,\chi_k}(n)(x^2-n^2)^{\rho}.\notag
\end{align}
On the other hand, by Lemma \ref{26},
{\allowdisplaybreaks
\begin{align}
&\df{x^{\rho+3/2}(p_1p_2\cdots p_k)^{\rho+1/2}}{2^k}
\sum_{\pm n_1,\pm n_2,\dots, \pm n_k}
\df{(-1)^{\textup{sgn}}K_{\rho+3/2}\left(\df{2^k\pi^kn_1n_2\cdots  n_k x}{p_1p_2\cdots p_k};\df12;k\right)}{(n_1n_2\cdots n_k)^{\rho+1/2}}\notag\\
=&\df{x^{\rho+3/2}(p_1p_2\cdots p_k)^{\rho+1/2}}{\phi(p_1)\phi(p_2)\cdots\phi(p_k)}
\sum_{n_1=1}^{\infty}\sum_{n_2=1}^{\infty}\cdots\sum_{n_k=1}^{\infty}
\df{K_{\rho+3/2}\left(\df{2^k\pi^kn_1n_2\cdots,  n_k x}{p_1p_2\cdots p_k};\df12;k\right)}{(n_1n_2\cdots n_k)^{\rho+1/2}}\notag\\
&\times\sum_{\chi_1 \text{ mod }p_1}\sum_{\chi_2 \text{ mod } p_2}\cdots\sum_{\chi_k \text{ mod } p_k}\chi_1(a_1)\chi_2(a_2)\cdots\chi_k(a_k)
\overline{\chi}_1(n_1)\overline{\chi}_2(n_2)\cdots\overline{\chi}_k(n_k)\notag\\
=&\df{x^{\rho+3/2}(p_1p_2\cdots p_k)^{\rho+1/2}}{\phi(p_1)\phi(p_2)\cdots\phi(p_k)}\sum_{\chi_1 \text{ mod }p_1}\sum_{\chi_2 \text{ mod } p_2}\cdots\sum_{\chi_k \text{ mod } p_k}\chi_1(a_1)\chi_2(a_2)\cdots\chi_k(a_k)\notag\\
&\times\sum_{n=1}^{\infty}d_{\overline{\chi}_1,\overline{\chi}_2,\dots,\overline{\chi}_k}(n)
\df{K_{\rho+3/2}\left(\df{2^k\pi^k n x}{p_1p_2\cdots p_k};\df12;k\right)}{n^{\rho+1/2}}\notag\\
=&\df{(p_1p_2\cdots p_k)^{\rho+1/2}}{\phi(p_1)\phi(p_2)\cdots\phi(p_k)}\sum_{\chi_1 \text{ mod }p_1}\sum_{\chi_2 \text{ mod } p_2}\cdots\sum_{\chi_k \text{ mod } p_k}\chi_1(a_1)\chi_2(a_2)\cdots\chi_k(a_k)\notag\\
&\times\sum_{n=1}^{\infty}nd_{\overline{\chi}_1,\overline{\chi}_2,\dots,\overline{\chi}_k}(n)\left(\df{x}{n}\right)^{\rho+3/2}
K_{\rho+3/2}\left(\df{2^k\pi^k n x}{p_1p_2\cdots p_k};\df12;k\right)\notag\\
=&\df{(p_1p_2\cdots p_k)^{\rho+1/2}}{\phi(p_1)\phi(p_2)\cdots\phi(p_k)}\sum_{\chi_1 \text{ mod }p_1}\sum_{\chi_2 \text{ mod } p_2}\cdots\sum_{\chi_k \text{ mod } p_k}\chi_1(a_1)\chi_2(a_2)\cdots\chi_k(a_k)\notag\\
&\times\df{\pi^{k\rho}}{(-i)^k\tau(\chi_1)\tau(\chi_2)\cdots\tau(\chi_k)(p_1p_2\cdots p_k)^{\rho-1/2}}\sideset{}{'}\sum_{ n\le x} n d_{\chi_1,\chi_2,\dots,\chi_k}(n)(x^2-n^2)^{\rho}\notag\\
=&\df{(-i)^k\pi^{k\rho}}{\phi(p_1)\phi(p_2)\cdots\phi(p_k)}\sum_{\chi_1 \text{ mod }p_1}\sum_{\chi_2 \text{ mod } p_2}\cdots\sum_{\chi_k \text{ mod } p_k}\chi_1(a_1)\chi_2(a_2)\cdots\chi_k(a_k)\notag\\
&\times \tau(\overline{\chi}_1)\tau(\overline{\chi}_2)\cdots\tau(\overline{\chi}_k)\sideset{}{'}\sum_{ n\le x} n d_{\chi_1,\chi_2,\dots,\chi_k}(n)(x^2-n^2)^{\rho},\label{a2}
\end{align}}
 where in the penultimate step we applied Theorem \ref{theorem1}, and in the last step, used the fact that for odd $\chi$  \cite[p.~45]{bew},
\begin{equation*}%\label{a3}
\tau(\chi_j)\tau(\overline{\chi}_j)=-p_j,\qquad 1\leq j\leq k.
\end{equation*}
If we now compare \eqref{a1} with \eqref{a2}, we deduce Theorem \ref{theorem2}.
\end{proof}

Next, we offer a generalization of Theorem \ref{ct}.

\begin{theorem}\label{omegak} Assume that $\chi_1, \chi_2, \dots, \chi_k$ are non-principal primitive odd characters modulo $p_1,p_2,\dots,p_k$, respectively. Recall the definition \eqref{sumchichi} of $\mathbb{D}^k(x)$.
%If  there exists a subset $n_r$, $1\leq n_r<\infty$, such that
%\begin{equation*}
%\sum_{n_r=1}^{\infty}\df{|\Realp\{\tau(\chi_1)\tau(\chi_2)\cdots
%\tau(\chi_k)n_r d_{\overline{\chi}_1,\overline{\chi}_2\,\dots,
%\overline{\chi}_k}(n_r)|}{n_r^{(3k+1)/(2k)}}=\infty,
%\end{equation*}
Then
\begin{align}
\varlimsup_{x\to\infty}\df{\Realp{\, {\mathbb{D}}^k(x)}}{x^{(3k-1)/(2k)}}&=+\infty, \label{aaa11}\\
\varliminf_{x\to\infty}\df{\Realp{\, {\mathbb{D}}^k(x)}}{x^{(3k-1)/(2k)}}&=-\infty. \label{abc11}
\end{align}
Both \eqref{aaa11} and \eqref{abc11} remain valid if we replace $\Realp$ by $\Imp$ in each of \eqref{aaa11} and \eqref{abc11}.
\end{theorem}

\begin{proof} Replacing $s$ by $s/2$ in \eqref{kfuncequa2}, we see that $A=\tf12 k$ and $\delta=3$.
Thus, from the definition \eqref{omega3},
$$\theta=\df{\tf12 k\cdot3-\tf12}{2\cdot\tf12 k}=\df{3k-1}{2k}.$$
With the use of Theorem \ref{omega1}, the remainder of the proof follows along the same lines as the proof of Theorem \ref{ct}.
\end{proof}
Recall from \eqref{S} that
\begin{gather*}
\mathbb{S}(a_1,a_2,\dots a_k;p_1,p_2,\dots p_k;x):=\mathbb{S}_{0}(a_1,a_2,\dots a_k;p_1,p_2,\dots p_k;x)\notag\\
=\sideset{}{'}\sum_{1\le n_1n_2\cdots n_k\le x} n_1n_2\cdots n_k\,\sin(2\pi n_1a_1/p_1)\sin(2\pi n_2a_2/p_2)\cdots \sin(2\pi n_ka_k/p_k).
\end{gather*}
From \eqref{a1} with $\rho=0$, we see that $\mathbb{S}(a_1,a_2,\dots, a_k;p_1,p_2,\dots, p_k;x)$ is a linear combination of terms of the form
$$\mathbb{D}^k(x)=\sideset{}{'}\sum_{n\leq x}n d_{\chi_1,\chi_2,\dots,\chi_k}(n).$$
In analogy with Conjecture \ref{3tttt}, a similar conjecture can be made for $$\mathbb{S}(a_1,a_2,\dots, a_k;p_1,p_2,\dots, p_k;x).$$
Thus, we apply Theorem \ref{omegak} to each of these sums to obtain the following conjecture.

\begin{conjecture}\label{3ttttt} If $\mathbb{S}(a_1,a_2,\dots a_k;p_1,p_2,\dots p_k;x)$ is defined by \eqref{S}, then
\begin{align*}
\varlimsup_{x\to\infty}\df{ \mathbb{S}(a_1,a_2,\dots, a_k;p_1,p_2,\dots, p_k;x) }{x^{(3k-1)/(2k)}}&=+\infty,\\ %\label{ccc111}\\
\varliminf_{x\to\infty}\df{ \mathbb{S}(a_1,a_2,\dots, a_k;p_1,p_2,\dots, p_k;x)}{x^{(3k-1)/(2k)}}&=-\infty. %\label{ccc222}
\end{align*}
\end{conjecture}

Note that Conjecture \ref{3tttt} is the special case $k=2$ of Conjecture \ref{3ttttt}.

Similarly, we can use \eqref{a1} and Theorem \ref{bigO} to obtain an upper bound for the order of $\mathbb{S}(a_1,a_2,\dots, a_k;p_1,p_2,\dots, p_k;x).$

\begin{theorem}\label{SbigOO} For every $\epsilon>0$, as $x\to\infty$,
\begin{equation}\label{SbigO}
 \mathbb{S}(a_1,a_2,\dots, a_k;p_1,p_2,\dots, p_k;x)=O\left(x^{2k/(k+1)+\epsilon}\right).
 \end{equation}
 \end{theorem}

\begin{proof}
Note that $\beta=2+\epsilon$ for each $\epsilon>0$.  Then, by \eqref{bigO3},
$$u=2+\epsilon-\df32-\df{1}{2k}=\df12+\epsilon-\df{1}{2k}.$$
Next,  by \eqref{bigO1},  we need to calculate
\begin{equation}\label{2ea}
\df12\delta-\df{1}{4A}+2Au\eta=\df32-\df{1}{2k}+k\left(\df12+\epsilon-\df{1}{2k}\right)\eta=\df32-\df{1}{2k}+\df{k-1}{2}\eta +k\epsilon\eta,
\end{equation}
where $\eta$ is a non-negative number to be determined.

Let $d_k(n)$ denote the number of ways $n$ can be written as a product of $k$ factors.  Then, by \eqref{d(n)bigO} and induction on $k$,
\begin{equation}\label{dkO}
{\sum_{n\leq x}}^{\prime}d_k(n)
=xP_{k-1}(\log x)+O(x^{(k-1)/(k+1)}\log^{k-1}x),
\end{equation}
where  $k\geq 2$ and $P_r(\log x)$ is a polynomial of degree $r$ in
$\log x$.  (See also \cite[p.~133, Equation (10.10)]{annals2}).
Next, use \eqref{dkO} and  partial summation to deduce that
\begin{equation}\label{3ea}
|\sum_{x< \mu_n\leq x+O(x^{1-1/k-\eta})}nd_{\chi_1,\chi_2,\dots,\chi_k}(n)|
\leq \sum_{x< \mu_n\leq  x+O(x^{1-1/k-\eta})}nd_k(n)
=O(x^{2-1/k-\eta}\log^{k-1} x).
\end{equation}
Appealing to Theorem \ref{bigO}, we should find the optimal value of $\eta$ by equating the powers in \eqref{2ea} and \eqref{3ea}.  Thus, we should solve
$$2-\df{1}{k}-\eta=\df32-\df{1}{2k}+\df{k-1}{2}\eta.$$
Hence,
$$\eta=\df{k-1}{k(k+1)},$$
and so the optimal power is, for every $\epsilon>0$,
$$2-\df{k-1}{k(k+1)}+\epsilon=\df{2k}{k+1}+\epsilon.$$
This completes the proof of \eqref{SbigO}.
\end{proof}

Note that if $k=2$, \eqref{SbigO} reduces to Theorem \ref{3tt}.
Note also that
\begin{equation*}\label{kk}
\df{2(k+1)}{k+2}-\df{2k}{k+1}=\df{2}{(k+1)(k+2)}
\end{equation*}
is the difference in the exponents of \eqref{SbigO} for successive values of $k$. Thus, increasing the number of $\sin$'s by 1 in $\mathbb{S}(a_1,a_2,\dots, a_k;p_1,p_2,\dots, p_k;x)$ increases the upper bound for the power in the error term by a ``small'' amount, i.e., $O(1/k^2)$.
%increases when the number of $\sin$'s in
 %$$\mathbb{S}(a_1,a_2,\dots, a_k;p_1,p_2,\dots, p_k)$$ increases by 1.

Suppose that $$ \df{a_j}{p_j}=\df14,\, 1\leq j\leq k,$$
so that the terms are equal to 0 if one or more of  the $n_j, 1\leq j\leq k$ are even. For odd $n_j$, let $n_j=2m_j+1,\, 1\leq j\leq k.$  Then,
\begin{align*}
&\sideset{}{'}\sum_{1\le n_1,n_2, \dots, n_k\le x} n_1 n_2\cdots n_k\,\sin(2\pi n_1a_1/p_1)\sin(2\pi n_2a_2/p_2)\cdots \sin(2\pi n_ka_k/p_k)\\
%\end{gather*}\end{document}
=&\sideset{}{'}\sum_{1\le 2m_1+1,2m_2+1,\dots, 2m_k+1\le x}(2m_1+1)(2m_2+1)\cdots(2m_k+1) (-1)^{m_1+m_2+\cdots+m_k}\\
=&\sideset{}{'}\sum_{1\le m_1,m_2,\dots, m_k\le x}(2m_1+1)(2m_2+1)\cdots(2m_k+1) (-1)^{m_1+m_2+\cdots+m_k},
\end{align*}
after replacing $x$ by $2x+1$.
Thus, Theorem \ref{theorem2} gives an identity for the weighted sum of products of positive odd lattice points $ 2m_1+1,2m_2+1,\dots, 2m_k+1$ in $k$-dimensional space weighted by $(-1)^{m_1+m_2+\cdots+m_k}$.  However, instead of applying Theorem \ref{theorem2} directly, if we apply Equation \eqref{a1} instead, we obtain a ``Big O'' bound for this sum weighted by $(-1)^{m_1+m_2+\cdots+m_k}$. Recall that we addressed the case $k=2$ earlier.

%\item The integral $K_{3/2}$ above can be evaluated in the same way as the integral in the following section.  Instead of using the Fourier $\cos$-transform for Meijer $G$-functions, one uses the Fourier $\sin$-transform instead.
     Finding a representation for a sum of $k$ $\cos$-functions appears to be enormously complicated.
    Similarly, finding a representation for a sum with a mixture of $\sin$'s and $\cos$'s appears also to be extremely complicated.  To make any progress, it would appear that the numbers of $\sin$'s and $\cos$'s should be equal.

\end{document}